\documentclass[a4paper,11pt]{article}
\usepackage[T1]{fontenc}
\usepackage[latin1]{inputenc}
\usepackage{color}
\usepackage{array}
\usepackage{amssymb}
\usepackage{amsmath}
\usepackage{hyperref}
\usepackage[all]{xy}
\usepackage{euscript}
\usepackage{mathrsfs}
\usepackage{graphicx}
\usepackage{xspace}
\usepackage{amsfonts}
\usepackage{multirow}
\usepackage{nicefrac}
\usepackage{textcomp}
\usepackage{url}
\usepackage{stmaryrd}
\usepackage{amsthm}
\usepackage{geometry}
\usepackage{mathtools}
\usepackage{epsfig}
\usepackage{txfonts}

\definecolor{purple}{rgb}{0.8,0.12,0.8}
\definecolor{orange}{rgb}{1.0,0.7,0.0}
\definecolor{pink}{rgb}{1,0.5,0.8}
\definecolor{blackg}{rgb}{0.1,0.25,0.1}
\definecolor{ForestGreen}{cmyk}{0.91,0,0.88,0.42}
\definecolor{Turquoise}{cmyk}{0.85,0,0.20,0}

\newcommand{\cB}{\mathcal{B}}
\newcommand{\cC}{\mathcal{C}}

\newcommand{\cE}{\mathcal{E}}
\newcommand{\cF}{\mathcal{F}}

\newcommand{\cH}{\mathcal{H}}

\newcommand{\cK}{\mathcal{K}}
\newcommand{\cL}{\mathcal{L}}

\newcommand{\cP}{\mathcal{P}}

\newcommand{\cR}{\mathcal{R}}

\newcommand{\cU}{\mathcal{U}}

\newcommand{\bH}{\mathbf{H}}

\newcommand{\ba}{\mathbf{a}}

\newcommand{\bc}{\mathbf{c}}

\newcommand{\bm}{\mathbf{m}}

\newcommand{\br}{\mathbf{r}}
\newcommand{\bs}{\mathbf{s}}

\newcommand{\bv}{\mathbf{v}}

\newcommand{\fB}{\mathfrak{B}}

\newcommand{\fS}{\mathfrak{S}}

\newcommand{\fb}{\mathfrak{b}}

\newcommand{\fd}{\mathfrak{d}}

\newcommand{\fr}{\mathfrak{r}}

\newcommand{\sB}{\mathscr{B}}

\newcommand{\sD}{\mathscr{D}}

\newcommand{\sG}{\mathscr{G}}
\newcommand{\sH}{\mathscr{H}}

\newcommand{\sP}{\mathscr{P}}

\newcommand{\sS}{\mathscr{S}}

\newcommand{\Z}{\mathbb{Z}}

\newcommand{\Q}{\mathbb{Q}}
\newcommand{\C}{\mathbb{C}}

\newcommand{\al}{\alpha}
\newcommand{\be}{\beta}
\newcommand{\si}{\sigma}
\newcommand{\la}{\lambda}
\newcommand{\ga}{\gamma}
\newcommand{\eps}{\varepsilon}

\newcommand{\ze}{\zeta}

\newcommand{\de}{\delta}

\newcommand{\La}{\Lambda}

\newcommand{\De}{\Delta}

\newcommand{\bla}{\boldsymbol{\la}}

\newcommand{\bmu}{\boldsymbol{\mu}}

\newcommand{\ts}{\tilde{s}}

\newcommand{\tbs}{\tilde{\bs}}

\newcommand{\hbs}{\hat{\bs}}

\newcommand{\ra}{\rightarrow}
\newcommand{\lra}{\longrightarrow}
\newcommand{\Ra}{\Rightarrow}

\newcommand{\eq}{\Leftrightarrow}
\newcommand{\eqdef}{\xLeftrightarrow{\text{def}}}

\newcommand{\mand}{\quad\text{and}\quad}
\newcommand{\st}{\;\,\text{such that}\;\,}

\newcommand{\Ue}{\mathcal{U}_q (\widehat{\mathfrak{sl}}_e)}
\newcommand{\bemptyset}{\boldsymbol{\emptyset}}
\newcommand{\ds}{\displaystyle}
\newcommand{\Hkn}{\bH_{k,n}}
\newcommand{\ug}{\Phi_{\bs}}
\newcommand{\llb}{\llbracket}
\newcommand{\rrb}{\rrbracket}

\newcommand{\Id}{\text{Id}}
\newcommand{\Irr}{\text{Irr}}
\newcommand{\preck}{\prec_{\cK}}

\newcommand{\llu}{\leq_\cU}

\newcommand{\Frac}{\text{Frac}}
\newcommand{\charac}{\text{char}}

\theoremstyle{plain}
\newtheorem{thm}{Theorem}[section]
\newtheorem{lem}[thm]{Lemma}

\newtheorem{cor}[thm]{Corollary}
\newtheorem{prop}[thm]{Proposition}

\theoremstyle{definition}
\newtheorem{defi}[thm]{Definition}

\theoremstyle{remark}
\newtheorem{exa}[thm]{Example}
\newtheorem{rem}[thm]{Remark}

\makeindex

\title{Generalised canonical basic sets for Ariki-Koike algebras}
\author{Thomas Gerber 
\footnote{\textsc{Laboratoire de Math\'ematiques et Physique Th\'eorique} (UMR 7350, CNRS - Universit\'e de Tours) \hspace{5cm}
Parc de Grandmont, 37200, Tours. 
E-mail address : \url{thomas.gerber@lmpt.univ-tours.fr} }}

\begin{document}
 
\maketitle

\abstract{
Let $\sH$ be a non semi-simple Ariki-Koike algebra.
According to \cite{Jacon2004} and \cite{GeckJacon2011}, 
there is a generalisation of Lusztig's $\ba$-function which induces a natural combinatorial order 
\linebreak (parametrised by a tuple $\bm$) on Specht modules.
In some cases, Geck and Jacon have proved that this order makes the decomposition matrix of $\sH$ unitriangular.
The algebra $\sH$ is then said to admit a "canonical basic set".
We fully classify which values of $\bm$ yield a canonical basic set for $\sH$ and which do not. 
When this is the case, we describe these sets in terms of "twisted Uglov" or "twisted Kleshchev" multipartitions.
}

\section{Introduction}

Over a field of characteristic $0$, the representation theory of the symmetric group $\fS_n$ is well-known.
In particular, thanks to Maschke's semi-simplicity criterion, it is sufficient to understand its irreducible representations.
In fact, they are parametrised by partitions of $n$, via an explicit bijection.
It is then possible to study the representations of $\fS_n$ in a combinatorial manner (using the notion of Young tableaux),
and deduce for instance the classic "hook-length formula" to compute the dimension of any irreducible representation.

In the so-called modular case, that is, when the ground field is of prime characteristic $e$, one loses many convenient properties,
notably the semi-simplicity property.
One can however collect some information. For instance, the irreducible representations 
are known to be parametrised by $e$-regular partitions of $n$ (that is, the partitions of $n$ with at most $e-1$ equal parts).
Also, the study of the "decomposition matrix", which measures the defect of semi-simplicity, is of great interest in modular representation theory.
This matrix is known to have a unitriangular shape with respect to the dominance order on the set of $e$-regular partitions, 
which parametrise its columns.
It is natural to try to extend this property to the more general groups $G(l,1,n)=(\Z/l\Z) \wr \fS_n$,
and to their quantizations, which are known as Ariki-Koike algebras.

Inspired by this typical example, Geck and Rouquier introduced in \cite{Geck1998} and \cite{GeckRouquier2001} 
the notion of canonical basic set for a non semi-simple Hecke algebra $\sH$.
This approach formalises the fact that the decomposition matrix $D$ of $\sH$ is unitriangular, with respect to a certain ordering of
its columns. Besides, it gives a bijection between $\Irr(\sH)$ and a subset $\sB$ of $\Irr(\cH)$, where $\cH$
is the semi-simple generic Hecke algebra which specialises to $\sH$.
The columns of $D$ are then parametrised by $\sB$.

In this paper, we focus on Ariki-Koike algebras, that is, Hecke algebras $\bH$ of the complex reflection groups $G(l,1,n)$.
When $\bH$ is non semi-simple, it is regarded as a specialisation,
parametrised by a pair $(e,\br)$ (where $e$ is an integer and $\bs$ is an $l$-tuple of integers),
of a generic Ariki-Koike algebra H.
The irreducible representations of H are known to be parametrised by $l$-partitions of $n$.
Using Brou\'e and Malle's "cyclotomic Hecke algebras", see \cite{BroueMalle1993},
Jacon defined in \cite{Jacon2004} $\ba$-invariants for Ariki-Koike algebras, depending on a parameter $\bm\in\Q^l$ (which itself depends on $\bH$),
which induce an order on the set of $l$-partitions of $n$.
These $\ba$-invariants are seen as a generalisation of Lusztig's $\ba$-function \cite{Lusztig1984}.
In \cite{GeckJacon2011}, Geck and Jacon showed compatibility between this order induced by the $\ba$-invariants, 
which has several geometric interpretations, and a combinatorial order $\ll_\bm$ defined using Lusztig's symbols.
It turns out that the order $\ll_\bm$ naturally arises in the study of $G(l,1,n)$, 
for instance in Kazhdan-Lusztig theory in type $B_n$ (that is, when $l=2$), see \cite{GeckIancu2012},
or via the representations of Cherednik algebras, as studied in \cite{ChlouverakiGordonGriffeth2011} or \cite{Liboz2012}.

Ariki's proof in \cite{Ariki1996} of the LLT conjecture \cite{LLT1995} enables us to compute the
decomposition numbers of $\bH$, when the ground field has characteristic zero, via the canonical basis of the Fock space $\cF_\br$ (Theorem \ref{arikistheorem}).
With this approach, Geck and Jacon showed (see \cite[Theorem 6.7.2]{GeckJacon2011}) that it is always possible to find a canonical
basic set $\sB$ for $\bH$ for an appropriate choice of the parameter $\bm$.
In fact, they showed that $\sB$ is the set $\Phi_\br(n)$ of Uglov $l$-partitions of rank $n$ associated to $\br$.

We study here the canonical basic set for $\bH$ with respect to $\ll_\bm$ in full generality: given an element $\bm\in\Q^l$,
we can define generalised $\ba$-invariants for $\bH$ depending on $\bm$.
In this setting, is there a canonical basic set for $\bH$ with respect to $\ll_\bm$?
After defining a certain finite set of hyperplanes $\sP^*\in\Q^l$,
we establish the following classification (see Theorem \ref{theorem}):
\begin{itemize}
\item If $\bm\notin\sP^*$, then there exists a canonical basic set, which we explicitely determine.
\item If $\bm\in\sP^*$, then there is no canonical basic set for $\bH$.
\end{itemize}
In the first case, the canonical basic set we describe can be regarded as a generalisation of the set of Uglov multipartitions determined by Geck and Jacon.
Hence we get other basic sets than the ones in \cite{GeckJacon2011}.

\medskip

The paper is structured as follows. 
In Section \ref{preliminaries} we recall some background on the representation theory of Ariki-Koike algebras
in the non semi-simple case. We introduce the order $\ll_\bm$ we use throughout this article and the notion of canonical basic set
in the sense of \cite{GeckRouquier2001} and \cite{Geck2007}.
We also formulate the precise question we are interested in.
Section \ref{existenceforappropriateparameters} summarizes the results of Geck Jacon in \cite[Chapters 5 and 6]{GeckJacon2011}
which are relevant for our purpose. We introduce the Fock space, which is a $\Ue$-module, and its highest weight submodule $V(\br)$
which appears in Ariki's theorem.
In Section \ref{mregular} we show, using Theorem \ref{bs}, that any so-called regular element $\bm$ yields
a canonical basic set, which we describe in terms of "twisted Uglov multipartitions".
Section \ref{masymptotic} is devoted to the asymptotic case, which roughly speaking corresponds to the case where the difference between
two arbitrary components of $\bm$ is large. After explaining the particularity of this setting,
we establish the existence of a canonical basic set with respect to $\ll_\bm$, namely the set of some "twisted Kleshchev multipartitions".
In the final Section \ref{msingular}, we prove that when $\bm$ is singular, that is, when $\bm\in\sP^*$, 
there is no canonical basic set with respect to $\ll_\bm$. We sum up these results in Theorem \ref{theorem}.

\medskip

\textbf{MSC:} 05E10, 20C08, 20C20, 16T30.

\section{Preliminaries}\label{preliminaries}

\subsection{General notations} \label{generalnotations}

We start with some notations about partitions and multipartitions.

\medskip

Let $n\in\Z_{\geq0}$ and $l\in\Z_{>0}$.
A \textit{partition} of $n$ is a decreasing sequence of non-negative integers $\la=(\la_1,\la_2,\dots)$ such that $\sum_i \la_i=n$. 
We write $|\la|=n$, the \textit{rank} of $n$.
The elements $\la_i$ are called the \textit{parts} of $\la$. We consider that a partition $\la$ has an infinite number of parts $\la_i=0$.
We denote by $\Pi(n)$ the set of partitions of $n$, and we write $\la\vdash n$ if $\la\in\Pi(n)$.
An \textit{$l$-partition} of $n$ (also referred to as a \textit{multipartition}) 
is a sequence of partitions $\bla=(\la^1,\la^2,\dots,\la^l)$ such that $|\la^1|+\dots+|\la^l| =n$.
We define $\Pi_l(n)$ the set of $l$-partitions of $n$, and we write $\la\vdash_l n$ if $\la\in\Pi_l(n)$.
The integer $|\bla|=n$ is called the rank of $\bla$.

Let $\bla\vdash_l n$. The \textit{Young diagram} of $\bla$ is the set $$[\bla]:=\{(a,b,c)\,;\, a\geq 1, c\in\llb 1, l\rrb, \text{and} \; 1\leq b \leq \la_a^c\}.$$
The elements of $[\bla]$ are called the nodes of $\bla$. 
For the sake of simplicity, we sometimes identify a multipartition with its Young diagram.
A node $\ga$ of $[\bla]$ is called a \textit{removable node} if the element $\bmu$ such that $[\bmu]=[\bla]\backslash\{\ga\}$ is still a multipartition (of rank $n-1$).
In this case, $\ga$ is also called an \textit{addable node} of $[\bmu]$.

A \textit{multicharge} is an element $\br=(r_1,\dots,r_l)\in\Z^l$. 
Given $e\in\Z_{>1}$, a multicharge $\br$ and a multipartition $\bla\vdash_l n$, we can associate to each node $\ga=(a,b,c)$ of $[\bla]$
its \textit{residue modulo $e$} $\fr_e(\ga):=r_c+b-a \mod e$.
For $i\in\llb 0, e-1 \rrb$, we call $\ga$ an \textit{$i$-node} if $\fr_e(\ga)=i$.

\medskip

We recall the classic \textit{dominance} order on $\Pi(n)$.
Let $\la=(\la_1,\la_2, \dots )$ and $\mu=(\mu_1,\mu_2,\dots)$ be two partitions of $n$. 
We say that that $\la$ \textit{dominates} $\mu$, and we write $\la\unrhd\mu$, if $\sum_{1\leq i \leq d} \la_i \geq \sum_{1\leq i\leq d}\mu_i$
for all $d\geq 1$.
More generally, we introduce a dominance order on the the set of sequences of rational numbers in the same manner.
Precisely, if $\al=(\al_1,\al_2,\dots)$ and $\be=(\be_1,\be_2,\dots)$ are such that $\sum_i \al_i =\sum_i \be_i$, we define $\al\unrhd\be$ by 
$\sum_{1\leq i \leq d} \al_i \geq \sum_{1\leq i\leq d}\be_i$ for all $d\geq 1$.
We write $\al\rhd\be$ if $\al\unrhd\be$ and $\al\neq\be$.

These are partial orders. In the following sections, we shall also consider other orders on the set $\Pi_l(n)$ of multipartitions.

\medskip
\medskip

We also need to introduce the \textit{extended affine symmetric group}.

\medskip

Let $l\in\Z_{>0}$. Denote by $\fS_l$ the symmetric group on $\llb 1, l \rrb$, and by $\si_i$, $i\in \llb 1, l-1 \rrb$ its generators
in its Coxeter presentation.
Let $\{y_1,\dots, y_l\}$ be the standard basis of $\Z^l$.
The extended affine symmetric group $\widehat{\fS}_l$ is the group with the following presentation:
\begin{itemize}
 \item Generators: $\si_i$, $i\in\llb 1, l-1\rrb$ and $y_i$, $i\in\llb 1, l\rrb$.
 \item Relations : \begin{itemize}
                       \item $\si_i^2=1$ for all $i\in\llb 1, l-1\rrb$,
                       \item $\si_i\si_{i+1}\si=\si_{i+1}\si_i\si_{i+1}$  for all $i\in\llb 1, l-2\rrb$,
                       \item $\si_i\si_j=\si_j\si_i$ whenever $i-j\neq 1 \mod l$,
                       \item $y_iy_j=y_jy_i$ for all $i,j\in\llb 1, l\rrb$,
                       \item $\si_i y_j=y_j\si_i$ for all $i\in\llb 1, l-1\rrb$ and $j\in\llb 1, l\rrb$ such that $j\neq i, i+1 \mod l$,
                       \item $\si_i y_i\si_i=y_{i+1}$ for all $i\in\llb 1, l-1\rrb$.
                    \end{itemize}

\end{itemize}

Note that $\widehat{\fS}_l$ is not a Coxeter group. Also, this group can be regarded as the semi-direct product $\Z^l\rtimes\fS_l$.

Given $e\in\Z_{>1}$, there is an action of $\widehat{\fS}_l$ on $\Z^l$, via the formulas:
\begin{itemize}
 \item $\si_i \br = (r_1, \dots, r_{i-1},r_{i+1},r_i, \dots, r_l)$ for all $i\in\llb 1, l-1\rrb$, and
 \item $y_i \br=(r_1,\dots,r_{i-1},r_i+e,r_{i+1},\dots, r_l)$ for all $i\in\llb 1, l\rrb$,
\end{itemize}
where $\br=(r_1,\dots,r_l)\in\Z^l$.

The set $\{\br\in\Z^l\,|\, 1\leq r_1\leq\dots\leq r_l\leq e\}$ is a fundamental domain for this action.

\subsection{Ariki-Koike algebras and decomposition maps}\label{akalgebras}

Let $l\in\Z_{>0}$ and $n\in\Z_{>1}$. Let $R$ be a subring of $\C$,
$u,V_1,\dots, V_l$ be independent indeterminates, and set $A=R[u^{\pm 1},V_1,\dots,V_l]$.

\begin{defi}
 The generic Ariki-Koike algebra $\bH_n=\bH_{A,n}(u,V_1,\dots,V_l)$ is the unital associative $A$-algebra with generators $T_i$,  $i=0,\dots,n-1$,
and relations \begin{itemize}
               \item $(T_i-u)(T_i+1)=0$ for all $i\in\llb 1, n-1 \rrb$.
               \item $(T_0-V_1)\dots(T_0-V_l)=0$
               \item $T_0T_1T_0T_1=T_1T_0T_1T_0$
               \item $T_i T_{i+1} T_i = T_{i+1} T_i T_{i+1}$ for all $i\in\llb 1, n-2 \rrb$
               \item $T_iT_j=T_jT_i$ whenever $|i-j|>1$
              \end{itemize}
Let $K$ be the field of fractions of $A$. We set $\bH_{K,n}=K\otimes_A\bH_n$.
We denote by $\Irr(\bH_{K,n})$ the set of irreducible $\bH_{K,n}$-modules.
\end{defi}

\begin{thm}[Ariki, \cite{ArikiKoike1994}]
The algebra $\bH_{K,n}$ is split semi-simple.
\end{thm}

As a consequence, Tits' deformation theorem implies that the irreducible representations of $\bH_{K,n}$ are in one-to-one correspondence with
the irreducible representations of the complex reflection group $W_n=G(l,1,n)$ over $\mathbb{K}=\Frac(R)$.
It is known that these representations are in one-to-one correspondence with $\Pi_l(n)$, the set of $l$-partitions of $n$. 
Therefore, we can write $$\Irr(\bH_{K,n})=\{E^{\bla}; \bla\vdash_l n\}.$$
The representations $E^{\bla}$, $\bla\vdash_l n$, are called the Specht representations.
The correspondence $ \Pi_l(n)\ra \Irr(\bH_{K,n}) , \bla \mapsto E^{\bla}$ is explicitely described in \cite[Section 3]{ArikiKoike1994}.

We have the following criterion of semi-simplicity for specialised Ariki-Koike algebras.

\begin{thm}[Ariki, \cite{Ariki2002}] \label{ss}
 Let $\theta: A \lra k$ be a specialisation, with $k=\Frac(\theta(A))$.  Denote $\eta=\theta(u)\neq0$ and $\eta_i=\theta(V_i)$, for all $i\in \llb 1,l\rrb$.
Then the specialised algebra $\Hkn(\eta,\eta_1,\dots,\eta_l)=k\otimes_A \bH_n$ is (split) semi-simple if and only if 
$$(\prod_{-n<d<n}\prod_{1\leq i<j\leq l}(\eta^d\eta_i-\eta_j))(\prod_{1\leq i\leq n}(1+\eta+\dots+\eta^{i-1}))\neq 0.$$
\end{thm}

\medskip

In his paper \cite{Mathas2004}, Mathas gives a thorough review of both the semi-simple and the modular representation theory of Ariki-Koike algebras.
In particular, we can recall this result by Dipper and Mathas:

\begin{thm}[Dipper and Mathas, \cite{DipperMathas2002}]
Let $\eta$ and $\eta_i$ ($i=1,\dots,l$) be as in Theorem \ref{ss}.
Denote $\cE=\{\eta_1,\dots,\eta_l\}$ and suppose that there exists a partition $\cE=\cE_1 \sqcup \dots \sqcup \cE_s$ such that 
$$\prod_{1\leq\al<\be\leq s} \prod_{(\eta_i,\eta_j)\in \cE_\al \times \cE_\be} \prod_{-n<N<n} (\eta^N\eta_i-\eta_j)\neq 0.$$
Then $\Hkn(\eta,\eta_1,\dots,\eta_l)$ is Morita equivalent to the algebra 
$$ \bigoplus_{\substack{n_1+\dots +n_s=n \\ n_1,\dots,n_s\geq 0}} \bH_{k,n_1}(\eta,\cE_1)\otimes_k \dots \otimes_k \bH_{k,n_s}(\eta, \cE_s).$$

\end{thm}

As a consequence, in order to study non semi-simple Ariki-Koike algebras, 
it is sufficient to consider the specialisations $\Hkn(\eta,\eta_1,\dots,\eta_l)$ of $\bH_{n}$ defined via $$\begin{array}{llll}
\theta : & A & \lra & k\\
 & u & \longmapsto & \ze \\
 & V_i & \longmapsto & \ze^{r_i}  ,
                    \end{array}$$ 
where $\ze$ is a primitive root of unity of order $e$ (possibly infinite),
and $r_i\in\Z$ for all $i\in\llb 1,l \rrb$. 
Hence each specialisation considered from now on will be characterised by a pair $(e,\br)$ where $e$ is the multiplicative order of $\ze=\theta(u)$,
and $\br=(r_1,\dots,r_l)\in\Z^l$.
We will denote equally $\Hkn(\ze,\ze^{r_1},\dots,\ze^{r_j})=\Hkn^{(e,\br)}$ the specialised Ariki-Koike algebra corresponding  to $(e,\br)$.
We also denote $\theta_{(e,\br)}$ the associated specialisation map.

\medskip

Consider the specialisation $\theta_{(e,\br)}:A\lra k$ with $k=\Frac(\theta_{(e,\br)}(A))$.
In accordance with \cite{Ariki1994}, \cite{Ariki1996} (or \cite{GeckPfeiffer2000} for $l=1,2$), 
there is an associated decomposition map $d_{\theta_{(e,\br)}} : R_0(\bH_{K,n}) \lra R_0(\Hkn^{(e,\br)})$, and we can write
$$d_{\theta_{(e,\br)}}([E^{\bla}])=\sum_{M\in\Irr(\Hkn^{(e,\br)})} d_{\bla,M}[M].$$
The \textit{decomposition matrix} of $\Hkn^{(e,\br)}$ is the matrix 
$$D_{\theta_{(e,\br)}}=(d_{\bla,M})_{\begin{subarray}{l} \bla\vdash_l n \\ M\in\Irr(\Hkn^{(e,\br)}) \end{subarray} }.$$
The elements $d_{\bla,M}$ are called the \textit{decomposition numbers} of $\Hkn^{(e,\br)}$.
If the specialised algebra is semi-simple, the decomposition map is trivial and $D_{\theta_{(e,\br)}}$ is the identity matrix.
In general, this matrix has a rectangular shape, since $|\Irr(\Hkn^{(e,\br)})|\leq|\{E^{\bla}; \bla\vdash_l n\}|$
(\cite{Ariki1994}, \cite{Ariki1996}).
For simplicity, we say that $\bla$ \textit{appears} in the column $C$ indexed by $M$ if $d_{\bla,M}\neq 0$.

\medskip

Actually, we can recover this matrix $D_{\theta_{(e,\br)}}$ from any specialisation $\Hkn^{(e,\bs)}$ where $\bs$ is in the class of $\br$ modulo $\widehat{\fS}_l$.
It is important to understand the consequences of choosing another multicharge to get the decomposition matrix. \label{important}

Denote by $\cC(\br)$ (or simply $\cC$) the class of $\br$ modulo $\widehat{\fS}_l$, 
and by $\cC_e(\br)$ (or simply $\cC_e$) the class of $\br$ modulo the subgroup $\left\langle y_1, ..., y_l \right\rangle$ of $\widehat{\fS}_l$
(i.e. the set of all $\bs=(s_1, ..., s_{l})\in\Z^l$ such that $\forall 1\leq i\leq l, s_i=r_i \mod e$).

\begin{itemize}
 \item  If $\bs\in\cC_e$, then it is clear that $\Hkn^{(e,\bs)}=\Hkn^{(e,\br)}$. 
Besides, the decompositions maps are the same. Therefore, the decomposition matrices are strictly equal.

\item If $\bs\in\cC$, one still has $\Hkn^{(e,\bs)}=\Hkn^{(e,\br)}$. We therefore denote $\Hkn$ this algebra.
However, the decomposition maps do not necessarily coincide. 
In fact, if we have $\bs=\si(\br)$, for some $\si\in\fS_l$, denote $\theta:=\theta_{(e,\br)}$ and $\theta_\si:=\theta_{(e,\bs)}$.
Then we can write $d_{\theta}([E^{\bla}])=\sum_{M\in\Irr(\Hkn)} d_{\bla,M}[M]$
and $d_{\theta_\si}([E^{\bla}])=\sum_{M\in\Irr(\Hkn)} d_{\bla,M}^\si[M]$.

Now since $\theta_\si(V_i)=\ze^{s_i}=\ze^{r_{\si(i)}}=\theta(V_{\si(i)})$, we have 
$$ d_{\bla,M}^\si=d_{\bla^\si,M}, \quad \forall \bla\vdash_l n\,,\,\forall M\in\Irr(\Hkn)$$
where $\bla^\si=(\la^{\si(1)},\dots\la^{\si(l)})$.

This means that the decomposition matrices are equal up to a permutation of the rows.
Equivalently, they are strictly equal (denoted by $D$) provided the parametrisation of the rows is changed: 
the row of $D$ labeled by $\bla$ with respect to the parametrisation yielded by $\br$ is labeled by $\bla^\si$
with respect to the parametrisation yielded by $\bs=\si(\br)$.

\end{itemize}

To sum up, the specialised Ariki-Koike algebra only depends on $\cC$. We denote it by $\Hkn$.
Hence we consider that for any $\bs\in\cC$, we obtain one genuine matrix, that we denote dy $D$, but that
each element $\bs\in\cC$ yields a different (in general) parametrisation of the rows of $D$.

In our purpose to study canonical basic sets, introduced in the next section, it is crucial to understand which parametrisation we use.
In fact, we will fix a multicharge $\br$ once and for all, and this will fix a parametrisation of the rows of $D$.

\medskip

In the sequel, we will be interested with the shape of this decomposition matrix.
One of the classic problems is to find an indexation of the simple $\Hkn$-modules so that $D$ is upper unitriangular, that is,
$$D=\begin{array}{cc}
\overbrace{\rule{19mm}{0pt}}^{\Irr(\Hkn)} & \\
\begin{pmatrix} 1 & 0 & \dots & 0 \\ \star & 1 & \dots & 0  \\ \vdots & \star & \ddots & \vdots \\ \vdots & \vdots & \ddots & 1 \\ 
\vdots & \vdots &  & \star \\ \vdots & \vdots & \hdots & \vdots  \end{pmatrix} & \left. \rule{0pt}{15mm} \right\} \Pi_l(n) \end{array} $$

More precisely, we ask for an order $\leq$ on $\Irr(\Hkn)$ such that $D$ has the above shape with respect to this order, that is,
$i>j \Ra M_i\leq M_j$, if $M_i\in\Irr(\Hkn)$ parametrises the $i$-th column of $D$.
These problems have been solved in \cite{GeckJacon2011}, using the theory of canonical basic sets.
This approach enables us find a bijection between $\Irr(\Hkn)$ and a subset of $\Pi_l(n)$, and therefore to label the simple $\Hkn$-modules by certain $l$-partitions,
the \textit{Uglov multipartitions}.
Accordingly, the orders used to parametrise the columns of $D$ are orders on $l$-partitions (i.e. on the rows of $D$).

In this paper, we address the "converse" question: given a certain natural order on the set of multipartitions,
is it possible to find a parametrisation of $\Irr(\Hkn)$ by a subset of $\Pi_l(n)$ such that $D$ is upper unitriangular with respect to this order?
We first need to precise which specific orders we are interested in, and some background about canonical basic sets.

\subsection{Canonical basic sets}\label{canonicalbasicsets}

One can associate to each simple $\bH_{K,n}$-module $E^{\bla}$ its Schur element $\bc^{\bla}$.
Explicit formulas for computing $\bc^{\bla}$ have been given independently in \cite{GeckIancuMalle2000} and \cite{Mathas2007}.
In \cite{ChlouverakiJacon2012}, Chlouveraki and Jacon have showed that $\bc^{\bla}$ is actually an element of
$\Z[u^{\pm1},V_1^{\pm 1},\dots , V_l^{\pm 1}]$ (that is, a Laurent polynomial in the variables $u,V_1,\dots,V_l$).

Now fix $\bm=(m_1,\dots,m_l)\in\Q^l$.
One can define the degree of $\bc^{\bla}$ by setting
$$ \deg (u^p V_1^{p_1} \dots V_l^{p_l})= p+m_1s_1+\dots+m_ls_l, \mand$$ 
$$\deg (\bc^{\bla})= \min \{ \deg(u^p V_1^{p_1} \dots V_l^{p_l}) \; ; \; u^{p}V_1^{p_1}\dots V_l^{p_l} \text{ is a monomial appearing in } \bc^{\bla} \}.$$
Such an element $\bm$ is then called a \textit{weight sequence}. 
Note that this definition of the degree is different from the usual one for Laurent polynomials (namely, one usually takes the maximum of the degrees of the monomials).

Extending Lusztig's \cite{Lusztig1984} definition of the $\ba$-function, one can then introduce \textit{ generalised $\ba$-invariants} for the modules $E^{\bla}$.
We simply set $\ba^{\bm}(\bla)=-\deg (\bc^{\bla})$.
The map $E^{\bla}\mapsto \ba^{\bm}(\bla)$ is then called a \textit{generalised $\ba$-function}, and coincides with Lusztig's $\ba$-function when $l=1$ and $\bm=m_1=1$.

\medskip

\begin{rem}\label{remarkcyclo}
The weight sequence $\bm$ we just introduced also has another algebraic meaning.
In \cite{BroueMalle1993}, Brou\'e and Malle have introduced the notion of "cyclotomic" Hecke algebra.
In the case of an Ariki-Koike algebra, see \cite[Chapter 5]{GeckJacon2011},
this is a one-parameter specialisation of $\bH_n$, parametrised by a pair $(\bm,t)\in \Q^l\times \Q$.
Thanks to Theorem \ref{ss}, any cyclotomic specialisation is known to be semi-simple.
Note that in \cite{GeckJacon2011}, the generalised $\ba$-function is defined on this cyclotomic specialisation.
Interestingly, any non semi-simple algebra $\Hkn^{(e,\br)}=\Hkn$ can be obtained by specialising a certain cyclotomic algebra.
In fact, if $m_i= r_i-e(i-1)/pl$ with $\gcd(p,e)=1$ and $t$ is such that $tm_i\in\Z$ for all $i$, 
we have a cyclotomic algebra $\bH_{\mathbb{K}(y),n}$ depending on an indeterminate $y$,
which can be specialised to $\Hkn$ via $y\mapsto \ze^{1/t}:=\exp(2ip\pi/et)$. 
In other terms, the following diagram commutes:

$$\xymatrix{ \bH_n \ar[dd]_{\theta_{(e,\br)}} \ar[rd]^{\theta_y} &  \\ & \bH_{\mathbb{K}(y),n} \ar[ld]^{\tilde{\theta}} \\ \Hkn &  }$$

where $\begin{array}[t]{llll}
 \theta_y :   & A & \lra & \mathbb{K}(y) \\
              & u & \longmapsto & y^t \text{\quad with $t$ such that $t(r_i-\frac{e(i-1)}{pl})\in\Z$}, \\
              & V_i & \longmapsto & y^{t(r_i-\frac{e(i-1)}{pl})}\xi_l^{i-1} \text{\quad for } i\in\llb 1,l\rrb 
                    \end{array}$
										
is the cyclotomic specialisation, where $\xi_l=\exp(2i\pi/l)$,
\medskip

and  $\begin{array}[t]{llll} \tilde{\theta} : & \mathbb{K}(y) & \lra & k \text{\quad such that } \ze^{1/t}\in k \\
                                           & y  & \longmapsto & \ze^{1/t}. \end{array}$
            
\end{rem}

\medskip

Now, the $\ba$-invariants induce an order on Specht modules, namely $E^{\bla}\sqsubseteq E^{\bmu} \eq [\bla=\bmu \text{ or } \ba^{\bm}(\bla)<\ba^{\bm}(\bmu)]$.
The general notion of canonical basic sets  requires an order on the Specht modules. 
In the case of Ariki-Koike algebras, it is natural to use this algebraic order.
In fact, we will use a combinatorial order $\ll_\bm$ which contains the order $\sqsubseteq$ above.
This is the order on shifted $\bm$-symbols defined in \cite{GeckJacon2011}.

\medskip

The \textit{shifted $\bm$-symbol} of $\bla=(\la^1,\dots,\la^l)\vdash_l n$ of size $p\in\Z$ is the $l$-tuple $\fB_\bm(\bla)=(\fB_\bm^1(\bla),\dots,\fB_\bm^l(\bla))$,
where $\fB_\bm^j(\bla)=(\fB_{p+\lfloor m_j \rfloor}^{j}(\bla),\dots,\fB_1^{j}(\bla))$, 
with $\fB_i^{j}(\bla)=\la_i^j-i+p+m_j$, for all $j\in\llb 1,l\rrb$ and $i\in\llb 1, p+\lfloor m_j \rfloor\rrb$.
Note that $p$ must be sufficiently large, so that $p+\lfloor m_j \rfloor \geq 1 + h^j$ for all $j\in\llb 1,l\rrb$, where $h^j=\max_{\la_i^j\neq 0} i$.
This ensures in particular that each $\fB_\bm^j(\bla)$ is well defined.
As usual, we consider that each partition $\la^j$ of $\bla$ has an infinite number of parts $\la_i^j=0$.

The shifted $\bm$-symbol $\fB_\bm(\bla)$ is pictured by an array whose $j$-th line (numbered from bottom to top) corresponds to $\fB_\bm^j(\bla)$.

\begin{exa}
 Let $\bm=(1/2,2,-1)$ and $\bla=(1.1,\emptyset, 2)\vdash_3 4$. We choose $p=3$. 
Then $$\fB_\bm (\bla)=\begin{pmatrix} 
                       0 & 3 & & & \\
                       0 & 1 & 2 & 3 & 4 \\
                       1/2 & 5/2 & 7/2 & 
                      \end{pmatrix}.$$

\end{exa}

Note that this symbol can easily be obtained from the shifted $\bm$-symbol of the empty $l$-partition, by adding the parts of $\la^i$ to the $i$-th row 
(numbered from bottom to top) of $\fB_\bm(\bemptyset)$, from right to left.

The symbol $\fB_\bm(\bla)$ has $h=lp+\sum_{1\leq j\leq l}\lfloor m_j\rfloor$ elements.

Write $\fb_\bm(\bla)=(\fb_\bm^1(\bla),\fb_\bm^2(\bla),\dots, \fb_\bm^h(\bla))$ the sequence of elements in $\fB_\bm(\bla)$, in decreasing order.
For $\bla, \bmu \in\Pi_l(n)$, we define the order $\ll_\bm$ by $$\bla \ll_\bm \bmu \quad \eqdef \quad \bla=\bmu \text{ or } \fb_\bm(\la) \rhd \fb_\bm(\bmu),$$
in the general sense of dominance order on sequences of rational numbers defined in Section \ref{generalnotations}.

\medskip

Set also $ n_\bm(\bla)=\sum_{1\leq i \leq h} (i-1)\fb_\bm^i(\bla).$
By \cite[Proposition 5.5.11]{GeckJacon2011}, we can compute the $\ba$-invariant of $\bla$ using symbols, namely
$\ba^{\bm}(\bla)=t(n_\bm(\bla)-n_\bm(\bemptyset))$.
As a direct consequence, we have the following compatibility property (\cite[Proposition 5.7.7]{GeckJacon2011}):
\begin{equation}\label{compatibility}[\bla \ll_\bm \bmu \text{ and } \bla\neq\bmu] \quad \Ra \quad \ba^{\bm}(\bla)<\ba^{\bm}(\bmu).\end{equation}


\medskip

This order on symbols has the advantage of being easier to handle, since it is purely combinatorial.
Besides, it naturally appears in the representation theory of the complex reflection groups of type $G(l,1,n)$.
For instance, when $l=2$ (i.e. in type $B$), Geck and Iancu showed in \cite[Theorem 7.11]{GeckIancu2012} that this order is compatible with
an order $\preceq_L$, defined (in \cite{Geck2009}) on $\Irr(G(l,1,n))$ using Lusztig's families
(and related to the order $\leq_{\cL\cR}$ defining Kazhdan-Lusztig cells).
In fact, they showed that in general, one has $\bla\preceq_L\bmu \Ra \bla\ll_\bm\bmu$, and that in some particular cases, both orders are equivalent.
Note that the version of the order $\ll_\bm$ defined in \cite[Section 3]{GeckIancu2012} is slightly different from the one we just defined.
Moreover, Chlouveraki, Gordon and Griffeth have used the compatibility property (\ref{compatibility}) above to deduce information on the 
decomposition of standard modules of Cherednik algebras, \cite[Theorem 5.7]{ChlouverakiGordonGriffeth2011}.
Also, Liboz showed in \cite{Liboz2012} that the order $\ll_\bm$ contains the order induced by the "$\bc$-function" on $\Irr(G(l,1,n))$
used in the representation theory of Cherednik algebras.

\medskip

We can now state the definition of a canonical basic set in the sense of \cite{Geck2007}, 
for both the order $\ll_\bm$ and the one induced by the $\ba$-invariants.

Fix $\br\in\Z^l$ and $e\in\Z_{>1}$.
Consider the specialised algebra $\Hkn^{(e,\br)}=\Hkn$. For $M\in\Irr(\Hkn)$, set $\sS(M)=\{\bla\vdash_l n \,|\, d_{\bla,M}\neq 0\}$.
Note that this set strongly depends on the choice of $\br$ (which is fixed once and for all), as explained on page \pageref{important}.

\begin{defi} \label{basicset} 
Assume that the following conditions hold:
\begin{enumerate}
 \item For $M\in\Irr(\Hkn)$, there exists a unique element $\bla_M\in\sS(M)$ such that for all $\bmu\in\sS(M)$, $\bla_M \ll_\bm \bmu$
(resp. $\ba^{\bm}(\bla_M)<\ba^{\bm}(\bmu)$ or $\bla=\bmu$).
 \item The map $\Irr(\Hkn)\ra\Pi_l(n)$, $M\mapsto \bla_M$ is injective.
 \item We have $d_{\bla_M,M}=1$, for all $M\in\Irr(\Hkn)$.
\end{enumerate}
Then the set $\sB:=\{\bla_M ; M\in\Irr(\Hkn)\}\subseteq \Pi_l(n)$ is in one-to-one correspondence with $\Irr(\Hkn)$. 
It is called a \textit{(generalised) canonical basic set for $(\Hkn, \br)$ with respect to $\ll_\bm$} (resp. \textit{with respect to the $\ba$-function}).
\end{defi}

\begin{rem} 
As a direct consequence, if there exists a canonical basic set for $(\Hkn,\br)$ with respect to $\ll_\bm$ (or with respect to the $\ba$-function), it is unique.
Moreover, the three conditions of Definition \ref{basicset} encode the fact that $D$ is upper unitriangular with respect to $\ll_\bm$.
\end{rem}

\begin{rem}
We have given here the definition of a canonical basic set for both the orders $\ll_\bm$ and the one induced by the $\ba$-function.
In this paper, we only consider the combinatorial order $\ll_\bm$, which is no real restriction because of the relation (\ref{compatibility}) between both orders.
However, both orders enjoying the same "continuity" property with respect to $\bm$ (see Lemma \ref{lemmsingular}) we give two versions of the main
result of this paper, namely Theorem \ref{theorem} (for the order $\ll_\bm$), and Theorem \ref{theorem0} (for the order induced by the $\ba$-function).
\end{rem}

\begin{rem}\label{importantremark}
Just like decomposition matrices, it is important to understand how the notion of canonical basic set depends on $\br$.
Indeed, this multicharge determines a parametrisation of the rows of $D$.
This parametrisation being invariant in the class $\cC_e$, we are ensured that for $\bs\in\cC_e$, 
if $(\Hkn,\bs)$ admits a canonical basic set $\sB$, 
then $\sB$ is the canonical basic set for $(\Hkn, \br)$.
However, this is not true for general $\bs\in\cC$. For such a multicharge,
it is sometimes possible to find a canonical basic set for $(\Hkn,\bs)$, even if $(\Hkn,\br)$ does not admit any canonical basic set
and both algebras are equal, see Example \ref{counterexample0}.
Also, if $\sB$ is the canonical basic set for $(\Hkn,\br)$, it is sometimes possible to find $\bs\in\cC$  
such that $(\Hkn,\bs)$ admits a canonical basic set $\sB'$ and $\sB'\neq\sB$, see Example \ref{exampledifferentbasicsets}.

\end{rem}

\medskip

Fix $\br\in\Z^l, n, l  \in \mathbb{Z}_{>1}, e \in \mathbb{Z}_{>1}$. 
The question of determining canonical basic sets for $(\Hkn,\br)$ has been solved in some cases.
First, in \cite{Jacon2004}, Jacon has studied the case where $m_i=r_i-e(i-1)/l$
(which is also when $\theta_{(e,r)}$ can be decomposed in a cyclotomic specialisation and a non semi-simple specialisation, as noticed in Remark \ref{remarkcyclo}),
and in \cite{GeckJacon2011}, Geck and Jacon have explained the more general case where $m_i=r_i-v_i$ with some restrictions on $(v_1,\dots,v_l)$.
We now want to fully review which values of $\bm\in\Q^l$ yield a canonical basic set for $(\Hkn,\br)$, and which do not.
We will see that unless $\bm$ belongs to some hyperplanes of $\Q^l$, the algebra $(\Hkn,\br)$ admits a canonical basic set with respect to $\ll_\bm$, 
which we can explicitely describe.

\medskip

\begin{rem}
Note that it is already known that canonical basic sets do not always exist. 
For instance, in level 2, that is, when $\Hkn$ can be seen as an Iwahori-Hecke algebra of type $B_n$,
Geck and Jacon have computed in \cite[Example 3.1.15 (c)]{GeckJacon2011} a decomposition matrix, associated to a specialisation (with $k=\mathbb{F}_2(v)$ and $n=2$)
which does not admit any canonical basic set with respect to the $\ba$-function.

Enlightened by \cite[Examples 5.7.3, 5.8.4 and 6.7.5]{GeckJacon2011}, 
we can then regard the cases of non-existence of a canonical basic set for $\Hkn$ (Proposition \ref{basicsetmsingular})
as anologues of this non-existence result in type $B_n$.

\end{rem}

\medskip

\begin{rem} 
The existence of canonical basic sets for Hecke algebras of more general reflection groups have been notably studied in \cite{ChlouverakiJacon2011} and
\cite{ChlouverakiJacon2012}.

\end{rem}

\medskip

\begin{rem}
We could have adressed a slightly different question.
Since we can recover the decomposition matrix from any $\bs\in\cC$ (up to a change of parametrisation of the rows), we could also ask
which weight sequences $\bm$ yield a canonical basic set for $(\Hkn,\bs)$, for some $\bs\in\cC$. 
Note that solving the first question automatically solves this weaker question, by taking the reunion over $\bs\in\cC$ of all weight sequences $\bm$
that yield  a canonical basic set for $(\Hkn, \bs)$. 
\end{rem}

\medskip

First, let us recall what particular values of $\bm$ are known to yield canonical basic sets for $(\Hkn,\br)$.

\section{Existence of canonical basic sets for appropriate parameters}\label{existenceforappropriateparameters}

Consider the specialised Ariki-Koike algebra $\Hkn=\Hkn^{(e,\br)}$ where $e\in\Z_{>1}$, and $\br=(r_1,\dots,r_l)\in\Z^l$.
We want to find a canonical basic set for $(\Hkn,\br)$, in the sense of Definition \ref{basicset}.
The following results prove that it is always possible to find $\bm\in\Q^l$ such that
$(\Hkn,\br)$ admits a canonical basic set with respect to $\ll_\bm$.
Besides, they can be explicitely described, either "directly" (FLOTW $l$-partitions) or recursively (Uglov $l$-partitions).
These results can be found in \cite{GeckJacon2011}.

\subsection{FLOTW multipartitions as canonical basic sets}

Let $\sS_e^l=\{ \br=(r_1,\dots,r_l)\in\Z^l\,|\, 0\leq r_j-r_i < e \text{ for all } i<j \}$.
In this section, we assume that $\br\in\sS_e^l$. 

The following definition is due to Foda, Leclerc, Okado, Thibon and Welsh, see \cite{FLOTW1999}.
\begin{defi}
 Let $\bla=(\la^1,\dots,\la^l)\vdash_l n$ and $\br\in\sS_e^l$.
Then $\bla$ is called a \textit{FLOTW $l$-partition} if:
\begin{enumerate}
 \item For all $j\in\llb 1, l-1\rrb$, $\la_i^j\geq \la_{i+r_{j+1}-r_j}^{j+1} \, ,\forall i\geq 1$; and $\la_i^l\geq \la_{i+e+r_1-r_l}^1\,,\forall i\geq 1$.
 \item The residues of the rightmost nodes of the length $p$ rows (for all $p>0$) of $\bla$ do not cover $\llb 0, e-1 \rrb$.
\end{enumerate}
Denote by $\Psi_{\br}$ the set of FLOTW $l$-partitions associated to $\br\in\sS_e^l$, and by $\Psi_{\br}(n)\subset~ \Psi_{\br}$ the ones of rank $n$.
\end{defi}

\begin{exa}
 If $l=3$, $e=4$ and $\br=(0,0,2)$, we have:
\begin{eqnarray*} \Psi_\br(3) & = &  \Big\{ \;  (3,\emptyset,\emptyset) \; ,\; (2,1,\emptyset) \;,\; (1,1,1) \;,\; (1.1,1,\emptyset) \;,\; 
(2.1 ,\emptyset,\emptyset) \;,\; (2,\emptyset,1) \;,\;  \\ & & (1,\emptyset,2) 
\;,\; (1.1,\emptyset,1) \;,\;  (1,\emptyset,1.1) \;,\; (\emptyset,\emptyset,2.1) \;,\; (\emptyset,\emptyset,3) \; \Big\}. \end{eqnarray*}
\end{exa}

\begin{rem}
 If $l=1$, the FLOTW $l$-partitions are exactly the $e$-regular partitions, that is, the partitions with at most $e-1$ equal parts.
\end{rem}

We have the following result by Geck and Jacon.

\begin{thm}[{\cite[Theorem 5.8.2]{GeckJacon2011}}]
 Let $\br\in\sS_e^l$, $\bv=(v_1,\dots,v_l)\in\Q^l$ such that $i<j\Ra 0<v_j-v_i<e$, and set $\bm=\br-\bv=(r_1-v_1,\dots,r_l-v_l)$.
 Then $(\Hkn,\br)$ admits a canonical basic set with respect to $\ll_\bm$, namely the set $\Psi_{\br}(n)$.
\end{thm}

Note that since $\sS_e^l$ contains a fundamental domain for the action of $\widehat{\fS}_l$ on $\Z^l$ (the one described in section \ref{generalnotations}), 
and since $\Hkn$ depends only on the class $\cC$ of $\br$ modulo $\widehat{\fS}_l$,
it is relevant to only consider the elements $\br\in\sS_e^l$.
Besides, this theorem holds regardless of the characteristic of the field $k$, and $\Psi_\br(n)$ has the advantage of being directly computable.

\subsection{Ariki's theorem and Uglov multipartitions as canonical basic sets} 

We now want to find a canonical basic set for $(\Hkn,\br)$ for an arbitrary value of $\br$.
In this subsection, we will need to assume that $\charac(k)=0$.
Indeed, this is an essential condition for Ariki's theorem to apply. 
His result links the theory of canonical bases (in the sense of Kashiwara, or Lusztig) of quantum groups with the modular representation theory
(in characteristic zero) of Ariki-Koike algebras.

\medskip

We do not recall here the theory of quantum groups. We refer to \cite{Ariki2002} and \cite{HongKang2002} for detailed background on $\Ue$ in particular.
We denote by $e_i, f_i, t_i, t_i^{-1}, \fd,\fd^{-1}$ the generators of $\Ue$,
and by $\omega_i$, $i\in\llb 0, e-1\rrb$, the fundamental weights of $\Ue$.

We redefine the Fock space and its properties.
Let $\br\in\Z^l$, $q$ be an indeterminate.
The Fock space $\cF_\br$ is the $\Q(q)$-vector space with formal basis $|\bla,\br\rangle$ where $\bla\vdash_l n$, i.e.
$$\cF_\br=\bigoplus_{n\in\Z_{\geq0}}\bigoplus_{\bla\vdash_l n} \Q(q) |\bla,\br\rangle.$$

We define an order on the set of addable or removable $i$-node of a $l$-partition $\bla$.
Let $\ga=(a,b,c)$ and $\ga'=(a',b',c')$ be two removable or addable $i$-nodes of $\bla \vdash_l n$.
We write $$\ga\prec_{(\br,e)} \ga' \text{ if }   \left\{      \begin{array}{l}      b-a+r_c<b'-a'+r_{c'} \text{ or }
\\  b-a+r_{c}=b'-a'+r_{c'} \mand c>c'.   \end{array}   \right.$$

Note that if $b_a+r_c=b'-a'+r_{c'}$ and $c=c'$, then $\ga$ and $\ga'$ are in the same diagonal of the same Young diagram,
so they cannot be both addable or removable.
Hence this order is well-defined.

Let $\bla\vdash_l n$ and $\bmu\vdash_l n+1$ such that $[\bmu]=[\bla]\cup\{\ga\}$ where $\ga$ is an $i$-node.
We set  $$\begin{aligned} N_i^{\prec}(\bla,\bmu)= &
\sharp\{\text{addable $i$-nodes $\ga'$ of $\bla$ such that $\ga'\prec_{(\br,e)} \ga$}\}-\\ &
\sharp\{\text{removable $i$-nodes $\ga'$ of $\bmu$ such that $\ga'\prec_{(\br,e)} \ga$}\},                                   
                                     \end{aligned}$$

$$\begin{aligned} N_i^{\succ}(\bla,\bmu)= &
\sharp\{\text{addable $i$-nodes $\ga'$ of $\bla$ such that $\ga'\succ_{(\br,e)} \ga$}\}-\\ &
\sharp\{\text{removable $i$-nodes $\ga'$ of $\bmu$ such that $\ga'\succ_{(\br,e)} \ga$}\},                                   
                                     \end{aligned}$$

$$N_i(\bla)=\sharp\{\text{addable $i$-nodes of $\bla$}\}-\sharp\{\text{removable $i$-nodes of $\bla$}\},$$

$$\mand \quad N_\fd(\bla)=\sharp\{\text{$0$-nodes of } \bla \}.$$

\label{espacedefock}

The following result is due to Jimbo, Misra, Miwa and Okado.
\begin{thm}[\cite{JMMO1991}] \label{jmmo} Let $\bla\vdash_l n$.
The formulas 
$$e_i\left. |\bla, \br\right\rangle = \sum_{\substack{\bmu\vdash_l n-1 \\ \fr_e([\bla]\backslash[\bmu])=i}} q^{-N_i^{\prec}(\bmu,\bla)}\left.|\bmu,\br\right\rangle,$$
$$f_i\left. |\bla, \br\right\rangle = \sum_{\substack{\bmu\vdash_l n-1 \\\fr_e([\bmu]\backslash[\bla])=i}} q^{-N_i^{\succ}(\bla,\bmu)}\left.|\bmu,\br\right\rangle,$$
$$t_i|\bla,\br\rangle = q^{N_i(\bla)}|\bla,\br\rangle \mand $$
$$\fd |\bla,\br\rangle = -(\De(\br)+N_\fd(\bla))|\bla,\br\rangle, \text{ for all $i\in\llb0,e-1\rrb$}$$
endow $\cF_\br$ with the structure of an integrable $\Ue$-module.
Here, $\De(\br)$ is a rational number defined in \cite{Uglov1999}.
\end{thm}

The element $|\bemptyset,\br\rangle\in\cF_\br$ is  a highest weight vector, of highest weight $\La_\br:=\omega_{r_1 \text{mod} e}+\dots + \omega_{r_l \text{mod} e}$.
We denote by $V(\br)\subset\cF_\br$ the irreducible highest weight $\Ue$-module spanned by $|\bemptyset,\br\rangle$.
This module $V(\br)$ is endowed with a crystal basis, a crystal graph $\sG_\br$, and a canonical (or global) basis, in the sense of \cite{Kashiwara1991}.
In order to determine the crystal graph $\sG_\br$, 
we first need to recall the definition of \textit{good addable} and \textit{good removable} $i$-nodes.

Let $\bla\vdash_l n$. Consider the set of its addable and removable $i$-nodes, ordered with respect to $\prec_{(\br,e)}$.
Encode each addable (resp. removable) $i$-node with the letter $A$ (resp. $R$). 
This yields a word of the form $A^{\al_1}R^{\be_1}\dots A^{\al_p}R^{\be_p}$. Delete recursively all the occurences of type $RA$ in this word.
We get a word of the form $A^\al R^\be$. Denote it by $w_i(\bla)$.
Let $\ga$ be the rightmost addable (resp. leftmost removable) $i$-node in $w_i(\bla)$. Then $\ga$ is called the \textit{good addable} (resp. \textit{good removable})
$i$-node of $\bla$.


\begin{defi}\label{uglov}
The set $\Phi_\br$ of \textit{Uglov $l$-partitions} is defined recursively as follows:
\begin{itemize}
 \item $\bemptyset\in\Phi_{\br}$,
 \item If $\bmu\in\Phi_{\br}$, then any $\bla$ obtained from $\bmu$ by adding a good addable node is also in $\Phi_{\br}$.
\end{itemize}
\end{defi}

Then we have the following result:

\begin{thm}[{\cite[Proposition 6.2.14]{GeckJacon2011}}] \label{recursivedefuglov}
The crystal graph $\sG_\br$ consists of vertices $|\bla,\br\rangle$ with $\bla\in\Phi_{\br}$
and arrows $|\bmu,\br\rangle \xrightarrow{\quad i \quad} |\bla,\br\rangle$ if and only if $[\bla]=[\bmu]\cup\{\ga\}$ where $\ga$ is the good addable $i$-node of $\bmu$.
\end{thm}

We can now state Ariki's theorem, proved in \cite{Ariki1996}.

Consider the canonical basis $\cB_\br$ of $V(\br)$.
It is indexed by Uglov $l$-partitions. We write $\cB_\br=\{ G(\bmu,\br)\,;\, \bmu\in\Phi_{\br} \}$.
Each element of $\cB_\br$ decomposes on the basis of $l$-partitions. Write $G(\bmu,\br)=\sum_{\bla\vdash_l n} c_{\bla,\bmu}(q)|\bla,\br\rangle.$

Let $\cB_\br^1$ be the specialisation of $\cB_\br$ at $q=1$, that is,
$$\cB_\br^1=\Big\{ \; \sum_{\bla\vdash_l n} c_{\bla,\bmu}(1)|\bla,\br\rangle\,;\, \bmu\in\Phi_{\br} \; \Big\}.$$

Recall that the elements $d_{\bla,M}\in\Z_{>0}, \bla\vdash_l n$, are the decomposition numbers associated to $M\in\Irr(\Hkn)$.
Define $$B(M,\br)=\sum_{\la\vdash_l n} d_{\bla,M} |\bla,\br\rangle. $$

\begin{thm}[{\cite[Theorem 12.5]{Ariki2002}}]\label{arikistheorem} Suppose that $\charac(k)=0$. Then
 $$\cB_\br^1=\Big\{ B(M,\br)\,;\, M\in\Irr(\Hkn),n\in\Z_{\geq0} \Big\}.$$
\end{thm}

Hence we have the following result concerning the decomposition matrix $D$ of $\Hkn$:

\begin{cor}
Set $C=(c_{\bla,\bmu}(1))_{\bla\vdash_l n,\bmu\in\Phi_{\br}}$. Then $C=D$ up to a reordering of the columns.
\end{cor}

In other words, if $\charac(k)=0$, it is sufficient to compute the canonical basis of the irreducible highest weight $\Ue$-module $V(\br)$ 
in order to recover the decomposition matrix $D$.

\medskip

Uglov, in \cite{Uglov1999}, determined a "canonical basis" $\cB_\br$ of $V(\br)$, generalising the work of Leclerc and Thibon in \cite{LeclercThibon1996}.
Another good reference is the thesis of Yvonne, \cite{Yvonne2005}.
This requires some theory about the \textit{affine Hecke algebra of type $A$} and \textit{$q$-wedge products}.
This approach permits us to establish the existence of a "canonical" basis of the whole Fock space, and also to compute $\cB_\br$.

We introduce one more notation. Let $\bs=(s_1,\dots,s_l)\in\cC$. For $\bm=(m_1,\dots,m_l)\in\Q^l$, we set $\bv=(v_1,\dots,v_l)=(s_1-m_1,\dots,s_l-m_l)$,
and we define $$\sD_\bs=\Big\{\bm\in\Q^l\,|\, i<j \Ra 0<v_j-v_i<e \Big\}.$$

Using Uglov's canonical basis of the Fock space, Geck and Jacon have proved the following result about canonical basic sets:

\begin{thm}[{\cite[Theorem 6.7.2]{GeckJacon2011}}] \label{bs} Suppose that $\charac(k)=0$.
Let $\br\in\Z^l$ and $\bm\in\Q^l$. If $\bm\in\sD_\br$, then $(\Hkn,\br)$ admits a canonical basic set with respect to $\ll_\bm$, namely the set $\Phi_{\br}(n)$.
\end{thm}

\medskip

In the rest of the paper, we will assume that $\charac(k)=0$, so that Ariki's theorem (and hence Theorem \ref{bs}) holds.

\medskip

We now wish to review the existence or non-existence of a canonical basic set for $(\Hkn,\br)$ with respect to $\ll_\bm$, 
depending on the values of $\bm$, and explicitely describe these sets when they exist.

In the following sections, we will denote by $\sP$ the following subset of $\Q^l$: 

$$\sP=\Big\{m\in\Q^l\,|\,\exists\, i\neq j \st (r_i-m_i) - (r_j-m_j) \in e\Z\Big\}.$$

Precisely, $\sP$ consists in the union of the hyperplanes 
$$\begin{aligned}\cP_{i,j}(\bs) &=\Big\{\bm\in\Q^l\,|\, (s_i-m_i)-(s_j-m_j)=0\Big\}\\ 
                                & =\Big\{\bm\in\Q^l\,|\, v_i-v_j=0\Big\} \quad (\text{where } v_i:=s_i-m_i \;\forall i\in\llb 1, l\rrb)\end{aligned}$$ 
over all $\bs\in\cC_e$ and all $1\leq i < j \leq l$.

Indeed, for $k\in\Z$, we have 

$$\begin{aligned} (r_i-m_i) - (r_j-m_j) =ke & \eq  (r_i-m_i) - (r_j+ke-m_j) =0\\
                                & \eq (s_i-m_i)-(s_j-m_j)=0  \\
                                & \eq v_i-v_j = 0 \end{aligned}$$
with $\bs=(s_1,\dots,s_i,\dots,s_j,\dots,s_l)=(r_1,\dots, r_i,\dots, r_j+ke, \dots,r_l)$.

Clearly, this is not a disjoint union, since $\cP_{i,j}(\bs)=\cP_{i,j}(\tbs)$ whenever $\ts_i=s_i+pe$ and $\ts_j=s_j+pe$ for some $p\in\Z$.
Also, when $l>2$, the hyperplanes $\cP_{i,j}(\bs)$ and $\cP_{i',j'}(\bs)$ always intersect, even when $(i,j)\neq(i',j')$.

\medskip

Now, for $\bm\in\Q^l$, we can always find a multicharge $\bs\in\cC_e$ which is "close" to $\bm$ in the following sense.
In $\Q^l$, consider the closed balls $B_{e/2}(\bs)$, with respect to the infinity norm, of radius $e/2$ and centered at $\bs$, for $\bs\in\cC_e$.
By definition of $\cC_e$, it is clear that 
$$ \bigcup_{\bs\in\cC_e} B_{e/2}(\bs) = \Q^l, \mand \bigcap_{\bs\in\cC_e} B_{e/2}(\bs) = \bigcup_{\bs\in\cC_e} \partial B_{e/2}(\bs).$$
In other terms, these balls cover $\Q^l$, and only their boundaries intersect.
Hence, there is a particular $\bs=(s_1,\dots,s_l)\in\cC_e$ such that $\bm\in B_{e/2}(\bs)$, and this multicharge is unique if $\bm$ is not on the boundary of the ball.
If it belongs to the boundary, then this means that there exists $i\in\llb 1, l \rrb$ such that $|v_i]=|s_i-m_i|=e/2$.
In this case, we make $\bs$ unique by setting $|v_i|=e/2$.

\begin{defi}\label{defadaptedmc}
The element $\bs$ thus obtained is called the \textit{$\bm$-adapted multicharge}.
\end{defi}

The following easy lemma will be useful in the last two sections.

\begin{lem}\label{lemminP} Let $\bm\in\Q^l$.
Suppose that $\bm\in\cP_{i,j}(\bs')$ for some $i<j$ and some $\bs'\in\cC_e$.
Then $\cP_{i,j}(\bs)=\cP_{i,j}(\bs')$, where $\bs$ is the $\bm$-adapted multicharge.
\end{lem}

\begin{proof}
The multicharge $\bs$ verifies in particular $0\leq |s_i-m_i|\leq e/2$ and $0\leq |s_j-m_j|\leq e/2$.
Also, because $\bs',\bs\in\cC_e$, we can write $s_i=s'_i+pe$ for some $p\in\Z$.
This gives $$0\leq |s'_i-m_i +pe|\leq e/2,$$ i.e. $$0\leq |s'_j-m_j +pe|\leq e/2 \text{ since } \bm\in\cP_{i,j}(\bs').$$
Hence we have $s_j=s'_j+pe$, which implies that $\cP_{i,j}(\bs)=\cP_{i,j}(\bs')$.
\end{proof}

As a consequence, if $\bm\in\sP$, it writes a priori $\ds\bm\in\bigcap_{\substack{(i,j)\in J \\ \bs'\in S}} \cP_{i,j}(\bs')$
for some index set $J$ and some $S\subset \cC_e$,
but, a posteriori, we can simply write $\ds\bm\in\bigcap_{(i,j)\in J} \cP_{i,j}(\bs)$, where $\bs$ is the $\bm$-adapted multicharge.

\section{Canonical basic sets for regular $\bm$}\label{mregular}

If $\bm \in \Q^l\backslash \sP$, we say that $\bm$ is \textit{regular}. In this section, we show that any regular $\bm$ defines an order $\ll_\bm$ 
with respect to which $\Hkn$ admits a canonical basic set. 
We use the fact that, for $\bs\in\cC_e$, if $\sB$ is the canonical basic set for $(\Hkn,\bs)$ with respect to $\ll_\bm$, 
then $\sB$ is the canonical basic set for $(\Hkn,\br)$ with respect to $\ll_\bm$ (see Remark \ref{importantremark}).
We first study the case $l=2$, and then the general case.

\subsection{$l=2$}

Here we have $\br=(r_1,r_2)\in\Z^2$, and $\sP$ is just a collection of parallel lines, namely the lines passing through $(s_1,s_2)\in\cC_e$ with slope 1.
The set $\sD_\bs$ is the domain strictly between the lines passing through $(s_1,s_2)$ and $(s_1,s_2-e)$, see Figure \ref{hyp1}.

\medskip

\textbf{Notation:} For $\bs=(s_1,s_2)\in\cC$, we denote $\tbs=(s_1,s_2+e)$.

\begin{figure}
\begin{center}
\scalebox{0.5}{\input{hyperplans10.pstex_t}}
\caption{The set $\sP$ and the domains $\sD_\bs$ in level $2$.}
\label{hyp1}
\end{center}
\end{figure}

\begin{prop}\label{basicsetmregularlevel2}

Let $\bm=(m_1,m_2)\in \Q^2\backslash\sP$. 
Then $(\Hkn,\br)$ admits a canonical basic set with respect to $\ll_\bm$, 
namely either $\Phi_{\bs}(n)$ or $\Phi_{\tbs}(n)$, where $\bs\in\cC_e$ is explicitely determined.
 
\end{prop}

\proof The idea is to show that any such $\bm$ belongs to a certain $\sD_{\hbs}$, for $\hbs\in\cC_e$.

Consider the $\bm$-adapted multicharge $\bs$ (cf. Definition \ref{defadaptedmc}).
It verifies, in particular,  $0\leq |s_i-m_i|\leq \frac{e}{2} \, , i=1,2.$.
We set, as usual, $v_i=s_i-m_i$, so that we have $0\leq |v_1-v_2|\leq e$.
The fact that $\bm$ is not located on a hyperplane of $\sP$ 
(and that $\bs$ is obtained from $\br$ after translation of each coordinate by an element of $e\Z$)
ensures that on can never have $v_1=v_2$.
Thus $0< |v_1-v_2|< e$.
Now, 
\begin{itemize}
 \item If $0< v_2-v_1< e$, then $\bm\in\sD_{\bs}$. 
By Theorem \ref{bs}, the set $\Phi_{\bs}(n)$ is the canonical basic set for the algebra $\Hkn^{(e,\bs)}$ with respect to the order $\ll_\bm$.
Therefore, by Remark \ref{importantremark}, $\ug(n)$ is the canonical basic set for $(\Hkn,\br)$ with respect to $\ll_\bm$.
\item If $0<v_1-v_2< e$, then $0<(v_2+e)-v_1<e$. 
Hence $\bm\in\sD_{\tbs}$ (where we recall that $\tbs=(s_1,s_2+e)$), 
so that by Theorem \ref{bs}, the set $\Phi_{\tbs}(n)$ is the canonical basic set for $(\Hkn,\tbs)$ with respect to the order $\ll_\bm$.
Since $\tbs\in\cC_e$, using Remark \ref{importantremark}, $\Phi_{\tbs}(n)$ is the canonical basic set for $(\Hkn,\br)$ with respect to $\ll_\bm$.
\end{itemize}
\endproof

\medskip

\begin{rem}\label{remarkmregularlevel2}
In level 2, the domains $\sD_{\bs},\; \bs\in\cC_e$, actually tile $\Q^2\backslash\sP$. 
In higher level this does not hold anymore, and we need to find other canonical basic sets than Uglov multipartitions.
\end{rem}

\medskip

Note that we can sometimes find different canonical basic sets for $(\Hkn,\br)$ and $(\Hkn,\bs)$ with respect to the same order $\ll_\bm$ if $\bs\in\cC\backslash\cC_e$,
as mentioned in Remark \ref{importantremark}.
This is what the following example shows.

\begin{exa}\label{exampledifferentbasicsets}
Let $l=2$, $e=4$, $\br=(1,0)$ and $\bs=\br^\si=(0,1)$ (where $\si=(12)$). Then $\bs\notin\cC_e$. Take $\bm=(1,-1)$.

Then $\bm\in\sD_\br$ since $\br-\bm=(0,1)$. Hence $\Phi_\br(n)$ is the canonical basic set for $(\Hkn,\br)$.
Besides, $\bm\in\sD_\bs$ since $\bs-\bm=(-1,2)$. Hence $\ug(n)$ is the canonical basic set for $(\Hkn,\bs)$,
but $\ug(n)\neq \Phi_\br(n)$ for $n\geq 2$ (which one can easily check).

\end{exa}

\medskip

\subsection{$l>2$}

Throughout this paper, we will use the following notations.
For $\boldsymbol{\al}=(\al_1,\dots \al_l)\in\Q^l$ and $\si\in\fS_l$, we denote $\boldsymbol{\alpha}^\si=(\al_{\si(1)},\dots \al_{\si(l)})$.
Similarly, for $\bla=(\la^1,...,\la^l)\vdash_l n$, we write $\bla^\si=(\la^{\si(1)},\dots,\la^{\si(l)})$.

Let $\bm=(m_1, ..., m_l)$ be an element of $\Q^l\backslash\sP$.

\begin{prop}\label{uglovpermute}
 Let $\bs\in\cC$ and $\si\in\fS_l$. If $\Phi_{\bs}(n)$ is the canonical basic set for $(\Hkn,\bs)$ with respect to $\ll_\bm$, 
then the set $$\si(\ug(n)):=\{\bla^\si\,;\, \bla\in\Phi_{\bs}(n)\}$$ of \textit{$\si$-twisted Uglov $l$-partitions}
is the canonical basic set for $(\Hkn,\bs^\si)$ with respect to $\ll_{\bm^\si}$.
\end{prop}

\proof
In order to prove this result, we need to define a \textit{twisted Fock space} $\cF_{\bs^\si}^\si$, which, as a vector space, is the Fock space $\cF_{\bs^\si}$, 
but has a $\si$-twisted $\Ue$-action.

\medskip

Recall that the action of $\Ue$ on the Fock space $\cF_\bs$ is derived from an order on the $i$-nodes of $l$-partitions. 
We define a twisted order on the removable and addable $i$-nodes of a multipartition in the following way : 
let $\ga=(a,b,c)$ and $\ga'=(a',b',c')$ be two removable or addable $i$-nodes of $\bla \vdash_l n$. 
We write $$\ga\prec_{(\bs^\si,e)}^\si \ga' \text{ if }   \left\{      \begin{array}{l}      b-a+s_{\si(c)}<b'-a'+s_{\si(c')} \text{ or }
\\  b-a+s_{\si(c)}=b'-a'+s_{\si(c')} \mand \si(c)>\si(c').     \end{array}   \right.$$
Now, if $\ga=(a,b,c)$ is a removable (resp. addable) $i$-node of $\bla=(\la^1,...,\la^l)$, 
then $\ga^\si:=(a,b,\si^{-1}(c))$ is a removable (resp. addable) $i$-node of $\bla^\si:=(\la^{\si(1)},...,\la^{\si(l)})$, so that we have
$$\ga^\si\prec_{(\bs^\si,e)}^\si \ga'^\si \eq \ga\prec_{(\bs,e)} \ga'.$$

\medskip

This order enables us to define the numbers $N_i^{\prec^\si}(\bla,\bmu)$ and $N_i^{\succ^\si}(\bla,\bmu)$. 
Let $\bla\vdash_l n$ and $\bmu\vdash_l n+1$ such that $[\bmu]=[\bla]\cup\{\ga\}$ where $\ga$ is an $i$-node. 
Then set  $$\begin{aligned} N_i^{\prec^\si}(\bla,\bmu)= &
\sharp\{\text{addable $i$-nodes $\ga'$ of $\bla$ such that $\ga'\prec_{(\bs^\si,e)}^\si \ga$}\}-\\ &
\sharp\{\text{removable $i$-nodes $\ga'$ of $\bmu$ such that $\ga'\prec_{(\bs^\si,e)}^\si \ga$}\}                                   
                                     \end{aligned}$$

and $$\begin{aligned} N_i^{\succ^\si}(\bla,\bmu)= &
\sharp\{\text{addable $i$-nodes $\ga'$ of $\bla$ such that $\ga'\succ_{(\bs^\si,e)}^\si \ga$}\}-\\ &
\sharp\{\text{removable $i$-nodes $\ga'$ of $\bmu$ such that $\ga'\succ_{(\bs^\si,e)}^\si \ga$}\}                                   
                                     \end{aligned}$$

We abuse the notation by denoting $\si$ the isomorphism of vector spaces
$$\si : \begin{array}{ccc} \cF_\bs & \lra & \cF_{\bs^\si} \\  \left. |\bla,\bs\right\rangle & \longmapsto & \left. |\bla^\si,\bs^\si\right\rangle \end{array}$$
Now we want do define a twisted action of $\Ue$ on $\cF_{\bs^\si}$. 

The action of $e_i$ and $f_i$, 
denoted by $e_i^\si.\left. |\bla^\si, \bs^\si\right\rangle$ and $f_i^\si.\left. |\bla^\si, \bs^\si\right\rangle$, are defined as follows:
 
$$\ds e_i^\si.\left. |\bla^\si, \bs^\si\right\rangle = 
\sum_{\fr_e([\bla^\si]\backslash[\bmu^\si])=i} q^{-N_i^{\prec^\si}(\bmu^\si,\bla^\si)}\left.|\bmu^\si,\bs^\si\right\rangle.$$ 
Then we have $$\begin{aligned}
e_i^\si.\left. |\bla^\si, \bs^\si\right\rangle & = \sum_{\fr_e([\bla]\backslash[\bmu])=i} q^{-N_i^{\prec}(\bmu,\bla)} \si(\left.|\bmu,\bs\right\rangle) \\
&= \si(\sum_{\fr_e([\bla]\backslash[\bmu])=i} q^{-N_i^{\prec}(\bmu,\bla)}\left.|\bmu,\bs\right\rangle)
  \end{aligned}$$
that is, $e_i^\si$ acts as $\si e_i\si^{-1}$.

\medskip

Similarly, if we set 
$$\ds f_i^\si.\left. |\bla^\si, \bs^\si\right\rangle = \sum_{\fr_e([\bmu^\si]\backslash[\bla^\si])=i} q^{-N_i^{\succ^\si}(\bla^\si,\bmu^\si)}\left.|\bmu^\si,\bs^\si\right\rangle,$$
we have 
$$\begin{aligned}
f_i^\si.\left. |\bla^\si, \bs^\si\right\rangle & = \sum_{\fr_e([\bmu]\backslash[\bla])=i} q^{-N_i^{\succ}(\bla,\bmu)} \si(\left.|\bmu,\bs\right\rangle) \\
&= \si(\sum_{\fr_e([\bmu]\backslash[\bla])=i} q^{-N_i^{\prec}(\bla,\bmu)}\left.|\bmu,\bs\right\rangle)
  \end{aligned}$$
that is, $f_i^\si$ acts as $\si f_i\si^{-1}$.

Hence by Theorem \ref{jmmo}, these new formulas, combined with the formulas 
$$t_i .\left. |\bla^\si, \bs^\si\right\rangle=q^{N_i(\bla)}|\bla^\si,\bs^\si\rangle$$
and $$\fd.\left. |\bla^\si, \bs^\si\right\rangle=-(\De(\bs)+N_\fd(\bla))|\bla^\si,\bs^\si\rangle$$
endow  $\cF_{\bs^\si}$ with the structure of an integrable $\Ue$-module, that we denote by $\cF_{\bs^\si}^\si$.

We continue the construction as in the non-twisted case. 
Denote by $V(\bs^\si)^\si$ the submodule of $\cF_{\bs^\si}^\si$ generated by the empty $l$-partition $\left. |\bemptyset, \bs^\si\right\rangle $. 
This is an irreducible highest weight $\Ue$-module for this twisted action, and the crystal basis of $V(\bs)$ is mapped to the one of $V(\bs^\si)^\si$
by the isomorphism $\si$. In particular the vertices of the crystal graph of $V(\bs^\si)^\si$ are the \textit{$\si$-twisted Uglov $l$-partitions}:
$$\si(\ug(n)):=\{(\la^{\si(1)},...,\la^{\si(l)})\,;\, (\la^1,...,\la^l)\in\Phi_{\bs}(n)\}.$$

It is an indexing set for the global basis of $V(\bs^\si)^\si$. Now this basis is also obtained from the global basis of $V(\bs)$ by applying $\si$.
That is, if we denote by $G^\si(\bla,\bs^\si)$ ($\bla\in\si(\ug(n))$) the elements of the canonical basis of $V(\bs^\si)^\si$, we have:
\begin{equation}\label{baseglobale} \si(G(\bla,\bs))=G^\si(\bla^\si,\bs^\si).\end{equation}

Write $\ds G^\si(\bla,\bs^\si)=\sum_{\bmu^\si\vdash_l n}d_{\bmu^\si,\bla^\si}(q)\left. |\bmu^\si,\bs^\si\right\rangle$ 
the decomposition of $G^\si(\bla,\bs^\si)$ on the basis of all $l$-partitions. By (\ref{baseglobale}), we have
$$\begin{aligned} 
&\si(\sum_{\bmu\vdash_l n}d_{\bmu,\bla}(q)\left. |\bmu,\bs\right\rangle ) & = \sum_{\bmu^\si\vdash_l n}d_{\bmu^\si,\bla^\si}(q)\left. |\bmu^\si,\bs^\si\right\rangle \\ 
\text{i.e. } & \sum_{\bmu\vdash_l n}d_{\bmu,\bla}(q)\left. |\bmu^\si,\bs^\si\right\rangle & = \sum_{\bmu\vdash_l n}d_{\bmu^\si,\bla^\si}(q)\left. |\bmu^\si,\bs^\si\right\rangle. 
\end{aligned}$$

Hence $\forall \bla\in\Phi_{\bs}(n), \,\forall\bmu\vdash_l n$, we have \begin{equation}\label{coeffdec} d_{\bmu^\si,\bla^\si}(q)=d_{\bmu,\bla}(q).\end{equation} 
In particular, this is true at $q=1$.

By Ariki's theorem (Theorem \ref{arikistheorem}), which holds for any realisation of the highest weight $\Ue$-module $V(\br)$,
the matrix $(d_{\bmu,\bla}(1))_{\begin{subarray}{l}\bmu\vdash_l n , \bla\in\Phi_{\bs}(n)\end{subarray}}$ is the decomposition matrix $D$ of $\Hkn$.
Hence by (\ref{coeffdec}), one can also parametrise the irreducible modules of $\Hkn$ by the elements of $\si(\ug(n))$
and recover the same matrix by labelling the $i$-th column by $\bla_i^\si\in\si(\ug(n))$ and the $j$-th line by $\bmu_j^\si\vdash_l n$.

Moreover, the fact that $\Phi_{\bs}(n)$ is the canonical basic set for $(\Hkn,\bs)$ with respect to $\ll_\bm$ means that $D$ is upper unitriangular with respect to $\ll_\bm$.
Since we have $$\bla_i\ll_\bm\bla_j\eq\bla_i^\si\ll_{\bm^\si}\bla_j^\si,$$
the matrix $D$ (with columns indexed by $\si(\ug(n))$ ) is upper unitriangular with respect to $\ll_{\bm^\si}$, 
i.e. $\si(\ug(n))$ is the canonical basic set for $(\Hkn,\bs^\si)$.

\endproof

We are now ready to prove the following general result.

\begin{prop}\label{basicsetmregular} 
Let $\bm=(m_1,...,m_l)\in\Q^l\backslash\sP$. Then $(\Hkn,\br)$ admits a canonical basic set with respect to $\ll_\bm$, namely $\si(\Phi_{\bs^{\si^{-1}}}(n))$, 
where $\bs$ is the $\bm$-adapted multicharge and $\si\in\fS_l$ is explicitely determined. 
We then say that $\si$ is the $\bm$-adapted permutation.
\end{prop}

\proof
As in the level 2 case, consider the $\bm$-adapted multicharge $\bs$, and set $v_i=s_i-m_i$, for all $i\in\llb 1, l \rrb$.
Since $\bm\notin\sP$, at most one coordinate $v_i$ can verify $|v_i|=\frac{e}{2}$. Moreover one can never have $|v_i-v_j|=0$ for $i\neq j$.
Hence, we have $i\neq j \Ra 0<|v_i-v_j|<e$ for all $i$, which implies that 
\begin{equation}\label{vpermute}
\text{there exists (a unique) } \tau\in\fS_l \st i<j \Ra 0<v_{\tau(j)}-v_{\tau(i)}<e.
\end{equation}

Since $\bm=\bs-\bv$, we have $\bm^\tau=\bs^\tau-\bv^\tau$. Because of (\ref{vpermute}), we see that $\bm^\tau\in\sD_{\bs^\tau}$, 
hence by Theorem \ref{bs}, $\Phi_{\bs^\tau}(n)$ is the canonical basic set for $(\Hkn,\bs^\tau)$. 
Since $\bs\in\cC_e$, $\bs^\tau\in\cC_e(\br^\tau)$ and therefore (using Remark \ref{importantremark} again), 
$\Phi_{\bs^\tau}(n)$ is the canonical basic set for $(\Hkn,\br^\tau)$ with respect to $\ll_{\bm^\tau}$.
Thus by Proposition \ref{uglovpermute}, $\tau^{-1}(\Phi_{\bs^\tau}(n))$ is the canonical basic set for $(\Hkn,\br)$ with respect to $\ll_\bm$.
Setting $\si=\tau^{-1}$, we get the result.

\endproof

\medskip

In the particular level 2 case, we thus have two different approaches which yield canonical basic sets.
Let $l=2$. Let $\bm\in\Q^2$, and take $\bs$ the $\bm$-adapted multicharge.
Denote $\si=(12)$ (in particular, $\si=\si^{-1}$).
Suppose that $\bm\notin\sD_{\bs}$. On the one hand, by Proposition \ref{basicsetmregular}, $\si(\Phi_{\bs^\si}(n))$ is the canonical basic set for $(\Hkn,\br)$.
On the other hand, we also have $\bm\in\sD_{\tbs}$, so that $\Phi_{\tbs}(n)$ is the canonical basic set for $(\Hkn,\br)$ 
(this is precisely Proposition \ref{basicsetmregularlevel2}).	

Hence one must have $\Phi_{\tbs}(n)=\si(\Phi_{\bs^\si}(n))$. In other terms, $$\Phi_{(s_1,s_2+e)}(n)=\{(\la^2,\la^1)\,;\, (\la^1,\la^2)\in\Phi_{(s_2,s_1)}(n)\}.$$
We recover a result by Jacon, namely \cite[Proposition 3.1]{Jacon2007}. 
However, in level $l>2$, the application $\bla \longmapsto \bla^\si$
is not necessarily a crystal isomorphism. Consequently, the canonical basic set $\si(\Phi_{\bs^{\si^{-1}}}(n))$ is not a priori a set of Uglov $l$-partitions.
However, we know exactly which of these applications are indeed isomorphisms between sets of some Uglov multipartitions.
Indeed, the crystal isomorphisms between the different sets of Uglov multipartitions (associated to $\bs\in\cC$) 
have been described by Jacon and Lecouvey in \cite{JaconLecouvey2010}.
In particular, \cite[Proposition 5.2.1]{JaconLecouvey2010} claims that 
$$\Phi_{(s_1,\dots,s_{l-1},s_l+e)}(n)=\{(\la^2,\dots,\la^l,\la^1)\,;\,(\la^1,\dots,\la^l)\in\Phi_{(s_l,s_1,\dots,s_{l-1})}(n)\}.$$
This proves that the application 
$$\begin{array}{cccc} \Phi_{\bs^{\si_0^{-1}}}(n) & \lra & \Phi_{\tbs}(n) & \\ \bla & \longmapsto & \bla^{\si_0} & \text{ with } \si_0=(1\;2\;\dots\;l) , \end{array}$$
where $\tbs=(s_1,\dots,s_{l-1},s_l+e)$, is a crystal isomorphism.
Applying several times $\si_0$ (which is of order $l$), we obtain $l-1$ different crystal isomorphisms $\bla\longmapsto\bla^{{\si_0}^k},\, k\in\llb 0,l-1\rrb$.

\begin{exa} Take $l=3$. Then ${\si_0}=(1\;2\;3)$ and ${\si_0}^2=(1\;3\;2)$. Then the following applications are crystal isomorphisms:

$$ \begin{array}{ccc} \Phi_{(s_3,s_1,s_2)}(n) & \lra & \Phi_{(s_1,s_2,s_3+e)} (n)\\ 
(\la^1,\la^2,\la^3) & \longmapsto & (\la^2,\la^3,\la^1)
\end{array},$$

$$ \begin{array}{ccc} \Phi_{(s_2,s_3,s_1)}(n) & \lra & \Phi_{(s_3,s_1,s_2+e)} (n)\\ 
(\la^1,\la^2,\la^3) & \longmapsto & (\la^2,\la^3,\la^1)
\end{array}.$$
\end{exa}

As in the level 2 case, it is possible to recover these results by looking at the domains $\sD_\bs$, $\bs\in\cC_e$.
Indeed, even though these domains do not tile $\Q^l$ (as already mentioned in Remark \ref{remarkmregularlevel2})
some weight sequences $\bm$ whose adapted permutation $\si$ verify $\si\neq\Id$ can also lie in a domain $\sD_{\hbs}$, for some $\hbs\in\cC_e$.
In that case we have two different constructions of the canonical basic set for $(\Hkn,\br)$, which must therefore coincide.
That is, for some values of $\si\in\fS_l$, the set $\si(\Phi_{\bs^{\si^{-1}}}(n))$ is necessarily a set of Uglov multipartitions $\Phi_{\hbs}(n)$, for some $\hbs\in\cC_e$.

Of course, it gets difficult to visualise the domains $\sD_\bs$ when $l\geq 3$.
Moreover, this argument does not hold whenever $\bla\longmapsto \bla^{\si}$ is not a crystal isomorphism between Uglov multipartitions.
In fact, these $\si$-twisted Uglov multipartitions yield in general new canonical basic sets for $(\Hkn,\br)$.

\medskip

\section{Canonical basic sets for asymptotic $\bm$}\label{masymptotic}

Let $\bs\in\cC_e$.
We will show that when the difference between the values of $\bs$ is large, the set of Uglov multipartitions stabilises 
(that is, no longer depends on the parameter $\bs$), 
and coincides with the set of Kleshchev multipartitions.
This is what we call the \textit{asymptotic} case, and such an $l$-tuple $\bs$ will be called asymptotic, see Definition \ref{defasymptotic}.

\subsection{Kleshchev multipartitions and asymptotic setting}\label{masymptoticdef}

\medskip

Let us recall in detail the relation between Uglov $l$-partitions and Kleshchev $l$-partitions. The Kleshchev $l$-partitions are defined
in the same manner as the Uglov $l$-partitions, except the order on $i$-nodes used to define an action of $\Ue$ on the Fock space is different.
Indeed, let $\ga=(a,b,c)$ and $\ga'=(a',b',c')$ be two removable or addable $i$-nodes of the same $l$-partition of $n$.
We define $$\ga\preck\ga' \eq \left\{ \begin{array}{l} c'<c \text{ \quad or} \\  c'=c \mand a'<a . \end{array}\right. $$

Note that this order only depends on the class $\cC_e$, not on some particular $\bs\in\cC_e$ anymore.

This permits us to give $\cF_\bs$ the structure of integrable $\Ue$-module via the same formulas used with $\prec_{(\bs,e)}$ in Theorem \ref{jmmo}.
We can then construct the crystal graph of the highest weight submodule spanned by $|\bemptyset,\bs\rangle$, in the same way as the Uglov multipartions (Theorem \ref{recursivedefuglov}).
Its vertices are labeled by what we call the \textit{Kleshchev $l$-partitions}. We denote by $\cK_{\cC_e}(n)$ the set of Kleshchev $l$-partitions of rank $n$.

Note that with this realisation as an $\Ue$-module, $\cF_\bs$ is actually a tensor product of Fock spaces of level $1$, see \cite{Uglov1999}.

\medskip

The following proposition connects both orders for certain values of $\bs$.

\begin{prop} \label{ug=kl} Let $\bs\in\cC_e$ such that $i<j \Ra s_i-s_j\geq n-e$ \footnote{Of course, this is equivalent to 
$s_i-s_{i+1}\geq n-e+1$ for all $i\in\llb 1 , l-1 \rrb$.}.
Then for all $m\leq n$, $\ug(m)=\cK_{\cC_e}(m)$. In particular , $\ug(n)=\cK_{\cC_e}(n)$.
\end{prop}

\proof It is sufficient to show that in this case, both orders on $i$-nodes are equivalent, i.e. $\ga\prec_{(\bs,e)}\ga'\eq\ga\preck\ga'$,
where $\ga=(a,b,c)$ and $\ga'=(a',b',c')$ are two removable or addable $i$-nodes of $\bla\vdash_l m$. 

Note that $-n\leq b'-a'-(b-a)\leq n$. Indeed, the difference between $b'-a'$ and $b-a$ is minimal if and only if ($\la^{c'}=\emptyset$ and $\la^c=(n)$)
or ($\la^{c'}=(1^n)$ and $\la^c=\emptyset$); and is maximal if and only if ($\la^{c'}=(n)$ and $\la^c=\emptyset$) or ($\la^{c'}=\emptyset$ and $\la^c=(1^n)$).

\medskip

First assume that $\ga\preck\ga'$. Then:

\begin{itemize} 
\item If $c'<c$, then  $s_{c'}-s_c\geq n-e+1$, hence $b'-a'+s_{c'}-(b-a+s_c)\geq -n+n-e+1 =-e+1$. 
Since $\ga$ and $\ga'$ have the same residue, this implies that
$b'-a'+s_{c'}$ and $b-a+s_c$ are congruent modulo $e$, thus $b'-a'+s_{c'}-(b-a+s_c)\geq 0$, and therefore $\ga\prec_{(\bs,e)}\ga'$.

\item If $c'=c$ and $a'<a$, then $b< b'$ since $\la^{c'}=\la^c$ is a partition and $\ga$ and $\ga'$ are on the border of $\la^c$.
Hence $b-a<b'-a'$, and $b-a+s_c< b'-a'+s_{c'}$, hence $\ga\prec_{(\bs,e)}\ga'$.
\end{itemize}

Conversely, assume that $\ga\prec_{(\bs,e)}\ga'$. Then:

\begin{itemize}
 \item If $b-a+s_c< b'-a'+s_{c'}$ then suppose $c'>c$. Then $s_c-s_{c'}\geq n-e+1$. 
Since $\ga$ and $\ga'$ have the same residue, we have $b'-a'+s_{c'}-(b-a+s_c)\geq e$, and thus $b'-a'-(b-a)\geq e+n-e+1=n+1$, whence a contradiction.
Hence $c'\leq c$. If $c'<c$ then $\ga\preck\ga'$, and if $c'=c$ then $b'-a'>b-a$ thus $a'<a$ for the same reason as before, and $\ga\preck\ga'$.
  \item If $b-a+s_c = b'-a'+s_{c'}$ and $c'<c$ then it is straightforward that $\ga\preck\ga'$.
\end{itemize}

The only difference (a priori) in the construction of the Uglov $l$-partitions on the one hand, and the Kleshchev $l$-partitions on the other hand
is the definition of the order on $i$-nodes.
Since we just proved that both orders coincide in this case, both sets are the same.

\endproof

\medskip

From this Proposition, we directly deduce:

\begin{cor} \label{corstabilisationkl}
Suppose $\bs\in\cC_e$. When the difference $s_i-s_j$, for all $i<j$, is sufficiently large, the set of Uglov multipartitions $\ug(n)$ stabilises,
and is equal to $\cK_{\cC_e}(n)$.
\end{cor}

\begin{rem} \label{remarkstabilisationkl}
Note that the bound $n-e+1$ is not necessarily sharp (even though it is an optimal condition for both orders on $i$-nodes to coincide),
it is a priori possible for Uglov multipartitions to stabilise at a weaker condition on $\bs$.
\end{rem}

Actually, the set of Uglov multipartitions stabilises in other directions, that is, under other conditions of $\bs$.
More precisely, we will show that they stabilise whenever the difference between any arbitrary coordinates of $\bs$
(without the condition $i<j$) is "large enough".

In order to describe this phenomenon, we introduce the set  of \textit{twisted Kleshchev multipartitions}.
Let $\pi\in\fS_l$. We define the \textit{$\pi$-twisted Kleshchev order} on $i$-nodes as follows: 
Let $\ga=(a,b,c)$ and $\ga'=(a',b',c')$ be two removable or addable $i$-nodes of the same $l$-partition of $n$.
We set $$\ga\preck^\pi\ga' \eq \left\{ \begin{array}{l} \pi(c')<\pi(c) \text{ \quad or} \\  \pi(c')=\pi(c) \mand a'<a \end{array}\right. $$

This just means that the lexicographic convention on the coordinates of the $l$-partition is twisted by $\pi$.
The \textit{$\pi$-twisted Kleshchev $l$-partitions} are then defined as in the non-twisted case (and as the in the "Uglov" case): 
they label the vertices of the crystal graph of the same highest weight module defined via the action of $\Ue$ derived from this order $\preck^\pi$.
We denote them by $\cK_{\cC_e}^\pi(n)$.

\medskip

\begin{rem} \label{remarkklsi} Note that it is equivalent to either build the set of $\pi$-twisted Kleshchev multipartitions associated to $\cC_e$, 
  or to twist via $\pi$ the set of Kleshchev multipartitions associated to $\cC_e(\br^{\pi^{-1}})$, i.e. 
$$\cK_{\cC_e}^\pi(n)=\pi(\cK_{\cC_e^{\pi^{-1}}} (n) ),$$
where $\cC_e^{\pi^{-1}}:=\cC_e(\br^{\pi^{-1}})$.
\end{rem}

\medskip

We have the following "asymptotic" property:

\begin{prop}\label{ug=klsi}
 Let $\bs\in\cC_e$ such that there exists $\pi\in\fS_l$ verifying $\pi(i)<\pi(j) \Ra s_i-s_j\geq n+1$.
 Then $\ug(m)=\cK_{\cC_e}^\pi(m)$ for all $m\leq n$. In particular, $\ug(n)=\cK_{\cC_e}^\pi(n)$.
\end{prop}

\proof
It is very similar to the one of Proposition \ref{ug=kl}. 
Indeed, we show that for $\ga=(a,b,c)$ and $\ga'=(a',b',c')$ two removable or addable $i$-nodes of $\bla\vdash_l m$, $\ga\prec_{(\bs,e)}\ga'\eq\ga\preck^\pi\ga'$.

\medskip

Assume that $\ga\preck^\pi\ga'$. Then:

\begin{itemize} 
\item If $\pi(c')<\pi(c)$, then  $s_{c'}-s_c\geq n+1$, hence $b'-a'+s_{c'}-(b-a+s_c)\geq -n+n+1 =1>0$. 
Hence $b'-a'+s_{c'}>(b-a+s_c)$ and $\ga\prec_{(\bs,e)}\ga'$.

\item If $\pi(c')=\pi(c)$ and $a'<a$. Then $c'=c$ since $\pi$ is a permutation.
Thus $b < b'$ since $\la^{c'}=\la^c$ is a partition and $\ga$ and $\ga'$ are on the border of $\la^c$.
Hence $b-a<b'-a'$, and $b-a+s_c< b'-a'+s_{c'}$, hence $\ga\prec_{(\bs,e)}\ga'$.
\end{itemize}

Conversely, assume that $\ga\prec_{(\bs,e)}\ga'$. Then:

\begin{itemize}
 \item If $b-a+s_c< b'-a'+s_{c'}$ then suppose $\pi(c')>\pi(c)$. Then $s_c-s_{c'}\geq n+1$, and thus $b'-a'-(b-a)> n+1$, whence a contradiction.
Hence $\pi(c')\leq\pi(c)$. If $\pi(c')<\pi(c)$ then 
$\ga\preck^\pi\ga'$, and if $\pi(c')=\pi(c)$ then $c'=c$ and $b'-a'>b-a$ thus $a'<a$ , and $\ga\preck^\pi\ga'$.
  \item If $b-a+s_c = b'-a'+s_{c'}$ and $\pi(c')<\pi(c)$ then $\ga\preck^\pi\ga'$.
\end{itemize}

Again, the only difference in the construction of the Uglov $l$-partitions on the one hand, and the $\pi$-twisted Kleshchev $l$-partitions on the other hand
is the definition of the order on $i$-nodes.
Since we have proved that both orders coincide in this case, these sets are the same.

\endproof

\medskip

Hence, we directly deduce the following stabilisation property, whenever the difference between two arbitrary coordinates of $\bs$ is large:

\begin{cor}\label{corstabilisationklsi}
 Let $\bs\in\cC_e$ and let $\pi\in\fS_l$. 
 When the difference $s_i-s_j$, for all $\pi(i)<\pi(j)$, is sufficiently large, 
then the set of Uglov $l$-partitions $\ug(n)$ stabilises, and is equal to $\cK_{\cC_e}^\pi(n)$.
\end{cor}

\medskip

Note that such a permutation $\pi$ verifies in particular $\pi(i) < \pi(j) \Ra s_i > s_j$.
We thus call $\pi$ the \textit{reordering permutation of $\bs$}.

\medskip

\begin{rem} As in Remark \ref{remarkstabilisationkl}, note that the bound is not necessarily sharp, 
and that Uglov multipartitions are likely to stabilise under weaker conditions.
In fact, when $\pi=\Id$, Proposition \ref{ug=klsi} gives a bound (namely $n+1$) on each $s_i-s_j$ beyond which $\ug(n)=\cK_{\cC_e}(n)$, 
but which is less precise than the one given in Proposition \ref{ug=kl} (namely $n-e+1$).
However, when $\pi\neq\Id$, the bound $n+1$ is optimal for the orders $\prec_{(\bs,e)}$ and $\preck^\pi$ to coincide.
\end{rem}

\medskip

\begin{rem}\label{remarkugsi=ug} Let $\bs$ and $\pi$ be as in Corollary \ref{corstabilisationklsi}, i.e. $\ug(n)=\cK_{\cC_e}^\pi(n)$.

It is important to notice that for all $\si\in\fS_l$, 
\begin{equation} \label{eq0} \Phi_{\bs^\si}(n)=\si(\cK_{\cC_e}^\pi(n)). \end{equation}
Indeed, this directly follows from the definition of the Kleshchev order on $i$-nodes. 
Since in this case $\ug(n)$ is a set of ($\pi$-twisted) Kleshchev multipartitions,
it is equivalent to either \begin{itemize}
                            \item twist the multicharge via $\bs\mapsto \bs^\si$ and build the corresponding Uglov crystal, or
                            \item twist via $\bla\mapsto\bla^\si$ these $\pi$-twisted Kleshchev $l$-partitions.
                           \end{itemize}

In other terms, replacing $\si$ by $\si^{-1}$, (\ref{eq0}) is equivalent to:

$$\si(\Phi_{\bs^{\si^{-1}}}(n))=\cK_{\cC_e}^\pi(n).$$
 
In particular, this shows that the canonical basic set $\si(\Phi_{\bs^{\si^{-1}}}(n))$ of Proposition \ref{basicsetmregular} is always equal to $\cK_{\cC_e}^\pi(n)$,
for any value of $\si\in\fS_l$.

\end{rem}

\medskip

We can now define \textit{asymptotic} multicharges and weight sequences.

\begin{defi}\label{defasymptotic} $\ $ \vspace{1mm}
 \begin{enumerate}
  \item Let $\bs\in\cC_e$. We say that $\bs$ is asymptotic if $\ug(n)=\cK_{\cC_e}^\pi(n)$ for some $\pi\in\fS_l$
  (in which case $\pi$ is the reordering permutation of $\bs$).
  \item Let $\bm\in\Q^l$. We say that $\bm$ is asymptotic if the $\bm$-adapted multicharge (see Proposition \ref{basicsetmregular}) is asymptotic.
 \end{enumerate}
\end{defi}

\medskip

\begin{rem}\label{remarkdefasymptotic}
According to Remark \ref{remarkugsi=ug}, $\bs$ is asymptotic if and only if for all $\si\in\fS_l$, \linebreak $\si(\Phi_{\bs^{\si^{-1}}}(n))=\cK_{\cC_e}^\pi(n)$.
\end{rem}

\medskip

Let us now focus on the question of the existence of canonical basic sets, given an asymptotic weight sequence $\bm$.
In the case where $\bm\notin\sP$, $\bm$ is regular, and we have already shown in Proposition \ref{basicsetmregular} 
that $(\Hkn,\br)$ admits a canonical basic set with respect to $\ll_\bm$, namely the set $\si(\Phi_{\bs^{\si^{-1}}}(n))$ where $\bs$ is the $\bm$-adapted multicharge
and $\si$ the $\bm$-adapted permutation. 
In virtue of Remark \ref{remarkugsi=ug}, these sets of $l$-partitions are all equal to $\cK_{\cC_e}^\pi(n)$.
We will show that in the remaining asymptotic cases,
the order $\ll_\bm$ yields a canonical basic set for $(\Hkn,\br)$ which is also a set of twisted Kleshchev multipartitions.

\subsection{Kleshchev multipartitions as canonical basic sets} \label{klmpascbs}

\medskip

Fix $\bm\in\sP$ such that $\bm$ is asymptotic.
Because of Lemma \ref{lemminP}, this means that $\bm\in\bigcup_{(i,j)\in J}\cP_{i,j}(\bs)$,
where $\bs$ is the $\bm$-adapted multicharge and is asymptotic, and where $J\subset\llb 1, l \rrb ^2$.

In order to understand the phenomenon that appears, it is interesting to keep in mind the results of Uglov in \cite{Uglov1999}.
They have also been reformulated in the thesis of Yvonne \cite{Yvonne2005}.
In his paper, Uglov defined a combinatorial order $\llu$ to study the matrix of the \textit{canonical basis of $\cF_\bs$}.
First, he showed that the Fock space can be endowed with a so-called (positive) canonical basis, which generalises the notion of canonical bases for integrable $\Ue$-modules.
The elements of this basis are indexed by $l$-partitions. 
Therefore, there is a transition matrix between this basis and the basis of $l$-partitions $\De$, whose rows and columns are indexed by $\Pi_l(n)$. 
It turns out that one can recover the decomposition matrix $D$ of the Ariki-Koike algebra $\Hkn$ by specialising $\De$ at $q=1$,
and by keeping only the columns indexed by Uglov multipartitions associated to $\bs$.
The interesting part is that this property holds for any multicharge $\bs\in\cC_e$ 
even though the matrices $\De$ associated to $\bs$ and $\bs'$ are different in general!

Moreover, Uglov proved that $\De$ is always unitriangular with respect to $\llu$, see \cite[Proposition 4.11]{Uglov1999}.
This implies that $(\Hkn,\br)$ has a canonical basic set with respect to $\llu$, namely the set of Uglov $l$-partitions.
Now, when $\bs$ is "asymptotic enough", one can show that the order $\ll_\bm$ is a refinement of $\llu$, that is
\begin{equation}\label{compatibilityorders} \bmu\llu\bla \Ra \bmu\ll_\bm \bla. \end{equation}
Thus, if $\bm\in\cP_{i,j}(\bs)$ for such an $\bs$, we are ensured that the set of Uglov multipartitions 
(which coincide with some $\pi$-twisted Kleshchev multipartitions)
is the canonical basic set for $(\Hkn,\br)$ with respect to $\ll_\bm$.

Unfortunately, this particular setting does not cover all the asymptotic cases. 
Indeed, the definition of an asymptotic weight sequence given in \ref{defasymptotic}
is not sufficient to deduce the compatibility property (\ref{compatibilityorders}) above.
However, using only combinatorial arguments, we can show the more general following result.

\begin{prop}\label{basicsetmasymptotic} Let $\bm\in\sP$ be an asymptotic weight sequence, let $\bs$ be the $\bm$-adapted multicharge.
Then $(\Hkn,\br)$ admits a canonical basic set with respect to $\ll_\bm$, namely the set $\cK_{\cC_e}^\pi(n)$, where
$\pi$ is the reordering permutation of $\bs$.
\end{prop}

\medskip

In order to prove this, we need the following technical lemma.
We introduce the following notation. Given a weight sequence $\bm$, $\eps>0$ and $I \subset \llb 1, l\rrb$, we define a new weight sequence by 
$\bm^{[\eps,I]}:=(m^{[\eps,I]}_i)_{i=1\dots l}$ where
$$m^{[\eps,I]}_i= \left\{ \begin{array}{l} m_i \text{ if } i\notin I \\ m_i+i\eps \text{ if } i\in I\end{array}\right. $$
\begin{exa}\label{examperturbe}
Take $l=3$ and $I=\{1,3\}$.
Then $\bm^{[\eps,I]}=(m_1+\eps,m_2,m_3+3\eps)$.
\end{exa}

\begin{lem} \label{lemmsingular} Let $\bm$ be an arbitrary weight sequence. Let $\bla,\bmu\vdash_l n$, $\bla\neq\bmu$.
Then there exists $\al_{\bla,\bmu}>0$ such that for all $\eps\in ]0,\al_{\bla,\bmu} [$,

\begin{enumerate}
\item $$ \bla\ll_\bm \bmu \quad \Ra \quad \left[ 
\begin{array}{c} \forall I \subset \llb 1, l\rrb, \text{ either \quad \quad  } \bla \ll_{\bm^{[\eps,I]}} \bmu \text{ \quad \quad or }   \\
\bla \text{ and } \bmu \text{ are not  } \text{ comparable with respect to } \ll_{\bm^{[\eps,I]}},   \;  \end{array} \right] $$

and

\item $$ \left[\begin{array}{c} \bla \text{ and } \bmu \text{ are not comparable} \\ \text{with respect to } \ll_\bm \end{array}\right] \Ra
\left[ \begin{array}{c} \bla \text{ and } \bmu \text{ are not comparable with} \\ \text{respect to } \ll_{\bm^{[\eps,I]}}, 
\; \forall I \subset \llb 1, l\rrb . \end{array}\right], $$

\end{enumerate}

\end{lem}

This means that for a small perturbation of $\bm$, the order $\ll_\bm$ never reverses:
at worst, $\bla$ and $\bmu$ become uncomparable.
Moreover, one can never gain comparability between multipartitions uncomparable with respect to $\ll_\bm$ when slightly perturbing $\bm$.

\medskip

\proof First, note that it is sufficient to prove these properties for the perturbations $\bm^{[\eps,k]}$ of $\bm$ defined by 
$\bm^{[\eps,k]}=(m_1,\dots, m_{k-1}, m_k+\eps, m_{k+1},\dots ,m_l)$, for all $k\in\llb 1,l \rrb$.
Indeed, the result then follows by induction, since $\bm^{[\eps,I]}$ is constructed by iterating this procedure.

\medskip

Recall that, by definition (Section \ref{canonicalbasicsets}), $\bla\ll_\bm \bmu$ and $\bla \neq \bmu$ means that $\fb_\bm(\bla) \rhd \fb_\bm(\bmu)$,
where $\fb_\bm(\la)=(\fb_\bm^1(\bla),\dots,\fb_\bm^h(\bla))$ is the decreasing sequence consisting of the elements of $\fB_\bm(\bla)$.

Let us first prove Assertion 1.

Let $\bla\ll_\bm \bmu$. For $k\in\llb 1, l \rrb$, consider the order $\ll_{\bm^{[\eps,k]}}$. It is obtained from $\ll_\bm$ simply by 
translating the $k$-th row of the symbols by $\eps$, and taking the dominance order on the decreasing sequences of these new elements.
Informally, when we choose $\eps$ to be "small", one cannot have $\bmu\ll_{\bm^{[\eps,k]}} \bla$.
Indeed, the fact that $\bla\ll_\bm \bmu$ creates a gap at some point between $\sum_i \fb_\bm^i(\bmu)$ and $\sum_i \fb_\bm^i(\bla)$, 
which cannot be recovered if $\eps$ is small enough.

Let us prove this properly. We also use a running example to illustrate the different points of the coming proof.
For simplicity, since $\bm\in\cP_{i,j}(\bs)$ for some $i,j$ and some $\bs\in\Z^l$, we can assume without loss of generality that $m_i$ and $m_j$ are integers.

\medskip

\textit{Example:} Take $e=2$, $l=3$, $n=38$, $\br=(1,0,0)$, $\bm=(3\nicefrac{1}{3},7,5)\in\cP_{2,3}(1,2,0)$. Let $\bla=(4.1,4^2.3.2^3.1,4^3.2.1)\vdash_l n$ 
and $\bmu=(4.2,4.3.2^5,5.4^2.1^2)\vdash_l n$.
The shifted $\bm$-symbols of $\bla$ and $\bmu$ of size $1$ are the following:

$$\fB_\bm(\bla)=\begin{pmatrix}
                 0 & 2 & 4 & 7 & 8 & 9 & & \\
                 0 & 2 & 4 & 5 & 6 & 8 & 10 & 11 \\
                 0\nicefrac{1}{3} & 1\nicefrac{1}{3} & 3\nicefrac{1}{3} & 7\nicefrac{1}{3} & & &
                \end{pmatrix}$$
and 
$$\fB_\bm(\bmu)=\begin{pmatrix}
                 0 & 2 & 3 & 7 & 8 & 10 & & \\
                 0 & 3 & 4 & 5 & 6 & 7 & 9 & 11 \\
                 0\nicefrac{1}{3} & 1\nicefrac{1}{3} & 4\nicefrac{1}{3} & 7\nicefrac{1}{3} & & &
                 \end{pmatrix}.$$
The corresponding sequences $\fb_\bm$ are 
$$\fb_\bm(\bla)=(11,10,9,8,8,7\nicefrac{1}{3},7,6,5,4,4,3\nicefrac{1}{3},2,2,1\nicefrac{1}{3},0\nicefrac{1}{3},0,0)$$  
and  $$ \fb_\bm(\bmu)=(11,10,9,8,7\nicefrac{1}{3},7,7,6,5,4\nicefrac{1}{3},4,3,3,2,1\nicefrac{1}{3},0\nicefrac{1}{3},0,0).$$

\medskip

Since $\bla\ll_\bm\bmu$ and $\bla\neq\bmu$, there exists a smallest integer $p$ such that $\fb_\bm^p(\bla)>\fb_\bm^p(\bmu)$.
Denote $\de=\fb_\bm^p(\bla)-\fb_\bm^p(\bmu)$.

In our example, $p=5$ and $\de=2/3$, since $\fb_\bm^i(\bla)=\fb_\bm^i(\bmu)\,\forall i<5$ and $\fb_\bm^{5}(\bla)=8$ and $\fb_\bm^{5}(\bmu)=7\nicefrac{1}{3}$.

\medskip

Now for all $i$, denote $\{m_i\}=m_i-\lfloor m_i \rfloor$ the fractional part of $m_i$, whenever $m_i\notin\Z$.
Set $\be_i=\min( \{m_i\}, 1- \{m_i\})$ (for all $m_i \notin\Z$), and $\be=\min_i \be_i$. If $m_i\in\Z$ for all $i$, then set $\be=1$. 
In particular $\be\leq\de$. In the example, $\be=1/3$.

Hence, set $0<\eps<\be$. 
Now, for all $k\in\llb 1, l\rrb$, consider the $\bm^{[\eps,k]}$-symbols of $\bla$ and $\bmu$.
In our example, for $k=3$ for instance, we get 
$$ \fB_{\bm^{[\eps,1]}}(\bla)=\begin{pmatrix}
                         0+\eps & 2+\eps & 4+\eps & 7+\eps & 8+\eps & 9+\eps & & \\
                         0 & 2 & 4 & 5 & 6 & 8 & 10 & 11 \\
                         0\nicefrac{1}{3} & 1\nicefrac{1}{3} & 3\nicefrac{1}{3} & 7\nicefrac{1}{3} & & &
                         \end{pmatrix}$$
and
$$\fB_{\bm^{[\eps,1]}}(\bmu)=\begin{pmatrix}
                        0+\eps & 2+\eps & 3+\eps & 7+\eps & 8+\eps & 10+\eps & & \\
                        0 & 3 & 4 & 5 & 6 & 7 & 9 & 11 \\
                        0\nicefrac{1}{3} & 1\nicefrac{1}{3} & 4\nicefrac{1}{3} & 7\nicefrac{1}{3} & & &
                        \end{pmatrix}.$$
 
%

Since $\eps<\be$, the "perturbed" elements (of $\fB_{\bm^{[\eps,k]}}$) are ordered in the same way as the original ones (those of $\fB_\bm$).
Precisely, for all $i$, we either have \begin{equation} \label{perturbation} \fb_{\bm^{[\eps,k]}}^i(\bla)=\left\{ \begin{array}{c}
                                                                             \fb_\bm^i(\bla) \text{\;\; or} \\
                                                                             \fb_\bm^i(\bla)+\eps,
                                                                               \end{array}\right. \end{equation}
and similarly for $\bmu$.

\medskip

Now, let $\al_{\bla,\bmu}=\min(\be,\de/p)$ and take $0<\eps<\al_{\bla,\bmu}$.
One can then compute $\sum_{i=1}^s \fb_{\bm^{[\eps,k]}}^i(\bla)$ and $\sum_{i=1}^s \fb_{\bm^{[\eps,k]}}^i(\bmu)$for all $s< p$.
Clearly, it is possible to have $\sum_{i=1}^s \fb_{\bm^{[\eps,k]}}^i(\bla) < \sum_{i=1}^s \fb_{\bm^{[\eps,k]}}^i(\bmu)$.
This is the case in the example, for $k=3$, since if we take $s=2$, 
we have $\fb_{\bm^{[\eps,1]}}^1(\bla)+\fb_{\bm^{[\eps,1]}}^2(\bla)=11+10<11+10+\eps=\fb_{\bm^{[\eps,1]}}^1(\bmu)+\fb_{\bm^{[\eps,1]}}^2(\bmu)$.
Hence, one can have $\bla\not\ll_{\bm^{[\eps,k]}}\bmu$.

However, we necessarily have:
\begin{itemize}
\item $\ds \sum_{i=1}^{p-1} \fb_{\bm^{[\eps,k]}}^i(\bmu) - \sum_{i=1}^{p-1} \fb_{\bm^{[\eps,k]}}^i(\bla) \leq (p-1)\eps$, and
\item $\fb_{\bm^{[\eps,k]}}^p(\bla) - \fb_{\bm^{[\eps,k]}}^p(\bmu) \geq \de-\eps$ since $\fb_\bm^p(\bla)- \fb_\bm^p(\bmu)=\de$.
\end{itemize}

Thus,  
$$\begin{aligned} \sum_{i=1}^p \fb_{\bm^{[\eps,k]}}^i(\bla) - 
\sum_{i=1}^p \fb_{\bm^{[\eps,k]}}^i(\bmu) & \geq -(p-1)\eps + \de - \eps \\
                                     & = -p\eps + \de \\
                                     & > -p\frac{\de}{p} + \de \text{\quad since \quad } \eps<\frac{\de}{p}\\
                                     & = 0. \end{aligned}$$

Hence one can never have $\bmu\ll_{\bm^{[\eps,k]}}\bla$, which proves the first point.

\medskip

\medskip

The proof of Assertion 2. is completely similar.
First, if $\bla$ and $\bmu$ are not comparable with respect to $\ll_\bm$, then there exist minimal integers $p_1$ and $p_2$ such that
$$\ds \sum_{i=1}^{p_1} \fb_\bm^i (\bla) > \sum_{i=1}^{p_1} \fb_\bm^i (\bmu) \mand \ds \sum_{i=1}^{p_2} \fb_\bm^i (\bla) < \sum_{i=1}^{p_2} \fb_\bm^i (\bmu).$$
We can assume without loss of generality that $p_1< p_2$.
We denote $\de_1= \fb_\bm^{p_1}(\bla) - \fb_\bm^{p_1}(\bmu)>0$ and $\de_2= \fb_\bm^{p_2}(\bmu) - \fb_\bm^{p_2}(\bla)>0$.
We take $\al_{\bla,\bmu}=\min\{ \be, \de_1/ p_1, \de_2 / p_2 \}$, where $\be$ is as in the proof of Assertion 1.
Again, because $\eps<\be$, we know that
$$\fb_{\bm^{[\eps,k]}}^i(\bla)=\left\{ \begin{array}{c}
                        \fb_\bm^i(\bla) \text{\;\; or} \\
                    \fb_\bm^i(\bla)+\eps,
                  \end{array}\right.$$

Now, on the one hand, we have
\begin{itemize}
\item $\ds \sum_{i=1}^{p_1-1} \fb_{\bm^{[\eps,k]}}^i(\bmu) - \sum_{i=1}^{p_1-1} \fb_{\bm^{[\eps,k]}}^i(\bla) \leq (p_1-1)\eps$, and
\item $\fb_{\bm^{[\eps,k]}}^{p_1}(\bla) - \fb_{\bm^{[\eps,k]}}^{p_1}(\bmu) \geq \de_1-\eps$ since $\fb_\bm^{p_1}(\bla)- \fb_\bm^{p_1}(\bmu)=\de_1$.
\end{itemize}

This gives
$$\begin{array}{rcl}\ds
\sum_{i=1}^{p_1} \fb_{\bm^{[\eps,k]}}^i(\bla) - 
\sum_{i=1}^{p_1} \fb_{\bm^{[\eps,k]}}^i(\bmu) & \geq & -(p_1-1)\eps + \de_1 - \eps  \\
                                       & = & -p_1\eps + \de_1   \\
                                     & > & -p_1\frac{\de_1}{p_1} + \de_1 \text{\quad since \quad } \eps<\frac{\de_1}{p_1} \\
                                     & = & 0.  
\end{array}$$

On the other hand, we have
\begin{itemize}
\item $\ds \sum_{i=1}^{p_2-1} \fb_{\bm^{[\eps,k]}}^i(\bla) - \sum_{i=1}^{p_2-1} \fb_{\bm^{[\eps,k]}}^i(\bmu) \leq (p_2-1)\eps$, and
\item $\fb_{\bm^{[\eps,k]}}^{p_2}(\bla) - \fb_{\bm^{[\eps,k]}}^{p_2}(\bmu) \leq -\de_2+\eps$ since $\fb_\bm^{p_2}(\bmu)- \fb_\bm^{p_2}(\bla)=\de_2$.
\end{itemize}

This gives
$$\begin{array}{rcl}\ds
\sum_{i=1}^{p_2} \fb_{\bm^{[\eps,k]}}^i(\bla) - 
\sum_{i=1}^{p_2} \fb_{\bm^{[\eps,k]}}^i(\bmu) & \leq & (p_2-1)\eps +  (-\de_2  + \eps)  \\
                                       & = & p_2\eps - \de_2   \\
                                     & < & p_2\frac{\de_2}{p_2} - \de_2 \text{\quad since \quad } \eps<\frac{\de_2}{p_2} \\
                                     & = & 0.  
\end{array}$$

This implies in particular that $\bla$ and $\bmu$ are not comparable with respect to the perturbed order $\ll_{\bm^{[\eps,k]}}$.

%

\endproof

The following corollary is then immediate.

\begin{cor} Under the assumptions of Lemma \ref{lemmsingular},
if $\bla \ll_{\bm^{[\eps,I]}} \bmu$, then $\bla \ll_\bm \bmu$.
\end{cor}

\medskip

\begin{figure}
 \begin{center}
  \scalebox{0.5}{\input{hyp50b.pstex_t}}
  \caption{The perturbation $\bm^{[\eps,\{1,2\}]}$ of $\bm$ in level 2.}
  \label{hyp4}
 \end{center}
\end{figure}

\medskip

\textit{Proof of Proposition \ref{basicsetmasymptotic}.} 
Recall that we have fixed a weight sequence $\bm$ which is asymptotic and belongs to $\sP$.
Denote then $\cP_{i,j}(\bs)$, for $(i,j)\in J$, the hyperplanes containing $\bm$, where $\bs$ is the  $\bm$-adapted multicharge and is asymptotic.
Set also $I=\{i,j \; ; (i,j)\in J \}$.
Since $\bs$ is asymptotic, we have $s_i\neq s_j$ for all $i\neq j$.
Let $\pi$ be the reordering permutation of $\bs$, that is, the element of $\fS_l$ verifying $[\pi(i)<\pi(j) \Ra s_i>s_j ]$.


Denote also $\de=\ds\min_{\substack{(i,j)\notin J, j\in I \\ \bs'\in\cC_e}}d(\bm,\cP_{i,j}(\bs')$,
that is the minimal distance (in the usual sense) between $\bm$ and the set of hyperplanes $\cP_{i,j}(\bs')$ with $j\in I$ but $(i,j)\notin J$
(hence it is positive).

For $M\in\Irr(\Hkn)$, write $\sS(M)=\{\bmu\vdash_l n \, |\, d_{\bmu,M}\neq 0 \}$, as in Definition \ref{basicset}.
Set $\ds \al=\min_{\bmu\in\sS(M)} (\al_{\bla,\bmu})$, where the elements $\al_{\bla,\bmu}$ are defined in Lemma \ref{lemmsingular},
and take $0<\eps<\min(\al,e/l,\de/l)$.

Consider now the perturbed weight sequence $\bm^{[\eps,I]}=(m^{[\eps,I]}_1,\dots, m^{[\eps,I]}_l)$.
Because $\bm \in \cP_{i,j}(\bs)$ for all $(i,j)\in J$, we have
\begin{equation}\label{mhyp} m_{j}-m_{i} = s_{j}-s_{i}. \end{equation}
Hence, for all $(i,j)\in J$, we have 
$$\begin{array}{rcl}
m^{[\eps,I]}_{j}-m^{[\eps,I]}_{i} & = & (m_{j}+j\eps) - (m_{i}+ i\eps) \\ 
                                      & = & m_{j} - m_{i} + (j-i)\eps \\
				    & = & s_{j}-s_{i} + (j-i)\eps.
\end{array}$$
We have $-l<j-i<l$. Because we have chosen $\eps<e/l$, we have $(j-i)\eps<e$.
But for $\bs',\bs''\in\cC_e$, we have $s''_i=s'_i+N_ie$ for some $N_i\in\Z$ and for all $i\in\llb 1, l \rrb$.
This implies that the weight sequence $\bm^{[\eps,I]}$ no longer belongs to any hyperplane of the form $\cP_{i,j}(\bs')$, with $\bs'\in\cC_e$ and $i,j\in I$.

Also, since $m^{[\eps,I]}_i = m_i$ for all $i\notin I$ and because of (\ref{mhyp}),
we know that $\bm^{[\eps,I]}$ does not belong to any hyperplane of the form $\cP_{i,j}(\bs)$ for all $i,j\notin I$ and for all $\bs\in\cC_s$.

Finally, consider a pair $(i,j)$ with $i\notin I$ and $j\in I$, so that $m^{[\eps,I]}_i = m_i$ and $m^{[\eps,I]}_{j}=m_{j}+j\eps$.
Then we have
$$\begin{array}{rcl}
m^{[\eps,I]}_{j}-m^{[\eps,I]}_{i} & = & m_{j}+j\eps - m_{i}+ i\eps \\ 
                                      & = & m_{j} - m_{i} + j\eps \\
				    & = & s_{j}-s_{i} + j\eps.
\end{array}$$
We have $j\leq l$, and since $\eps<\de/l$, we have $j\eps<\de$.
Therefore, the new weight sequence does not $\bm^{[\eps,I]}$ belong to any hyperplane of the form $\cP_{i,j}(\bs')$ with $j\in I$, $(i,j)\notin J$ and 
$\bs'\in\cC_e$.

To sum up, we have just proved that $\bm^{[\eps,I]}\notin\sP$.

Therefore, by Proposition \ref{basicsetmregular}, $(\Hkn,\br)$ admits a canonical basic set with respect to $\ll_{\bm^{[\eps,I]}}$,
namely a set of twisted Uglov $l$-partitions, which is equal to the set $\cK_{\cC_e}^\pi (n)$ (see Remark \ref{remarkugsi=ug} for instance),
where $\pi$ is the reordering permutation of $\bs$.
Denote $\sB=\cK_{\cC_e}^\pi (n)$.

Since $\sB$ is the canonical basic set with respect to $\ll_{\bm^{[\eps,I]}}$, there exists a unique $\bla\in\sB$ verifying:
\begin{equation}\label{minimalelement}
\begin{array}{c} \bla\ll_{\bm^{[\eps,I]}} \bmu \quad \forall \bmu\in\sS(M).
\end{array}
 \end{equation}

\medskip

Suppose that $\sB$ is not the canonical basic set for $(\Hkn,\br)$ with respect to $\ll_\bm$.
Then either: 
\begin{enumerate}
 \item there exists $\bmu\in\sS(M)$ such that $\bmu\ll_\bm \bla$. Because $\eps<\al$, Lemma \ref{lemmsingular} applies.
In particular, Assertion 1. ensures that one can never find $I$ such that $\bla \ll_{\bm^{[\eps,I]}} \bmu$,
which contradicts (\ref{minimalelement}).
 \item there exists $\bla' \neq \bla$ such that $\bla'$ is also minimal in $\sS(M)$ with respect to $\ll_\bm$.
In this case, by Point 2. of Lemma \ref{lemmsingular}, $\bla$ and $\bla'$ are not comparable with respect to $\ll_{\bm^{[\eps,I]}}$. 
Therefore, $\bla$ and $\bla'$ are both minimal with respect to $\ll_{\bm^{[\eps,I]}}$, which contradicts (\ref{minimalelement}).
\end{enumerate}

\begin{flushright}
 $\square$
\end{flushright}

\section{Canonical basic sets for singular $\bm$} \label{msingular}

Denote by $\sP^*$ the set of all $\bm$ in $\sP$ such that $\bm$ is not asymptotic. If $\bm\in\sP^*$, we say that $\bm$ is \textit{singular}. 

In the previous section, we have considered perturbations $\bm^{[\eps,I]}$ of $\bm$.
In this section we will need more general perturbations.
In fact, for $\rho\in\fS_l$, $I\subset\llb, 1 , l \rrb$ and $\eps>0$, 
we define the weight sequence $\bm^{[\eps,I,\rho]}=(m^{[\eps,I,\rho]}_1,\dots,m^{[\eps,I,\rho]}_l)$ by:
$$m^{[\eps,I,\rho]}_i= \left\{ \begin{array}{l} m_i \text{ if } i\notin I \\ m_i+\rho(i)\eps \text{ if } i\in I\end{array}\right. $$

\begin{rem}
Note that, in particular, $\bm^{[\eps,I,\Id]}=\bm^{[\eps,I]}$.
\end{rem}

In this section, since $\bm$ is singular, the $\bm$-adapted multicharge $\bs$ is non-asymptotic and $\bm$ belongs to $\bigcup_{(i,j)\in J}\cP_{i,j}(\bs)$
for some $J$.
Recall that if we set $I=\{i,j\, ;(i,j)\in J\}$ as in Section \ref{klmpascbs}, we have, for all $(i,j)\in J$
\begin{equation}\label{id}
m^{[\eps,I,\Id]}_j-m^{[\eps,I,\Id]}_i = s_j-s_i + (j-i)\eps
\end{equation}

One can now consider the perturbations $\bm^{[\eps,I,(ij)]}$, for $(i,j)\in J$ (that is, associated to the transposition $(ij)$.
They verify 
\begin{equation}\label{transp}
m^{[\eps,I,(ij)]}_j-m^{[\eps,I,(ij)]}_i = s_j - s_i + (i-j)\eps
\end{equation}

Looking at (\ref{id}) and (\ref{transp}), we see that $\bm^{[\eps,I,\Id]}$ and $\bm^{[\eps,I,(ij)]}$ are on opposite sides of $\cP_{i,j}(\bs)$.
We are now ready to prove the following proposition.

\medskip

\begin{prop}\label{basicsetmsingular}
Let $\bm$ be a singular weight sequence. Then $(\Hkn,\br)$ does not admit any canonical basic set with respect to $\ll_\bm$.
\end{prop}

\proof Suppose that there exists a canonical basic set $\sB$ for $(\Hkn,\br)$ with respect to $\ll_\bm$.

For $M\in\Irr(\Hkn)$, recall that we have denoted $\sS(M)=\{ \bmu\vdash_l n \,|\, d_{\bmu,M}\neq 0\}$.
By definition, there exists a unique element $\bla_M \in \sS(M)$ such that $\bla_M \ll_\bm \bmu$ for all $\bmu\in\sS(M)$.

We follow the same notation as in the proof of Proposition \ref{basicsetmasymptotic},
and take $0<\eps<\min(\al,\be)$.
Then, for the same reason as in that proof, $\bm^{[\eps,I,\rho]}$ is regular for all $\rho\in\fS_l$.
Hence by Proposition \ref{basicsetmregular},
there exists a canonical basic set $\sB^{[\rho]}$ for $(\Hkn,\br)$ with respect to $\ll_{\bm^{[\eps,I,\rho]}}$,
namely the set of some twisted Uglov $l$-partitions.
Since $\bs$ is not asymptotic, Remark \ref{remarkdefasymptotic} implies that there exists $\rho_1$ and $\rho_2$ such that
 \begin{equation}\label{differentbs}
 \sB^{[\rho_1]} \neq \sB^{[\rho_2]}.
 \end{equation}
Note that this is true for $\rho_1=\Id$ and $\rho_2=(ij)$ for some $(i,j)\in J$ because of the remark following (\ref{id}) and (\ref{transp}).
Since $\sB^{[\rho_1]}$ is the canonical basic set with respect to $\ll_{\bm^{[\eps, \rho_1]}}$,
there exists a unique element $\bla^{[1]}_M$ such that for all $\bmu\in\sS(M)$, $\bla^{[1]}_M \ll_{\bm^{[\eps,\rho_1]}} \bmu$.
Similarly, there exists a unique element $\bla^{[2]}_M$  such that for all $\bmu\in\sS(M)$, $\bla^{[2]}_M \ll_{\bm^{[\eps,\rho_2]}} \bmu$.

Now by (\ref{differentbs}), there exists a particular $M_0\in\Irr(\Hkn)$ such that 
\begin{equation} \label{differentmin} \bla^{[1]}_{M_0} \neq \bla^{[2]}_{M_0}.\end{equation}

Thus, we have:
\begin{itemize}
 \item $\bla^{[1]}_{M_0} \ll_{\bm^{[\eps,\rho_1]}} \bla_{M_0}$ and $ \bla_{M_0} \ll_{\bm} \bla^{[1]}_{M_0} $. 
But by Lemma \ref{lemmsingular} (which applies since $\eps<\al$), 
this not possible if $\bla_{M_0} \neq \bla^{[1]}_{M_0}$.
Hence $\bla_{M_0}=\bla^{[1]}_{M_0}$.
 \item $\bla^{[2]}_{M_0} \ll_{\bm^{[\eps,\rho_2]}} \bla_{M_0}$ and $ \bla_{M_0} \ll_{\bm} \bla^{[2]}_{M_0} $. 
Again, by Lemma \ref{lemmsingular}, this not possible if $\bla_{M_0} \neq \bla^{[2]}_{M_0}$.
Hence $\bla_{M_0}=\bla^{[2]}_{M_0}$.
\end{itemize}

Hence, $\bla^{[1]}_{M_0}=\bla^{[2]}_{M_0}$, which contradicts (\ref{differentmin}).
\endproof

\medskip

As previously mentioned in Remark \ref{importantremark}, a singular weight sequence $\bm$ can however yield a canonical basic set for $(\Hkn,\bs)$,
but with $\bs\in\cC\backslash\cC_e$ (i.e. with a different parametrisation of the rows of $D$). This is illustrated in the following example.
 
\begin{exa}\label{counterexample0}
Let $l=2$, $e=3$, $n\geq 4$, $\br=(1,0)$. In particular $\br$ is not asymptotic, which one can check directly by
computing $\Phi_\br(n)$ and $\cK_{\cC_e}(n)$.
Take $\bs=\br^{\si}=(0,1)$ (where $\si=(12)$), and $m=(0,-1)$.
Then $\bm\in\cP_{1,2}(\br)$ since $\br-\bm=(1,1)$, and by Proposition \ref{basicsetmsingular}, $(\Hkn,\br)$ 
does not admit any canonical basic set with respect to $\ll_\bm$.
However, $\bs-\bm=(0,2)$, so that $\bm\in\sD_\bs$. By Proposition \ref{basicsetmregularlevel2}, $(\Hkn,\bs)$ admits a canonical basic set with respect to $\ll_\bm$,
namely the set $\ug(n)$.
\end{exa}

\medskip

\begin{rem}
In the particular case where $l=2$ and $e=\infty$ (i.e. when $\ze$ is not a root of unity, cf. Section \ref{akalgebras}), one can use a simpler argument
to show that there is no canonical basic set.
First, note that in this case, $\cC_e=\{\br\}$, and $\sP$ consists of just the line passing through $\br$ with slope one. Also, $\sP^*=\sP$.
There is a "natural" symbol  which encodes the weight of a multipartition $\bla$ seen as an element of the Fock space $\cF_\br$.
Because $\bm$ is singular, this information is precisely the data carried by the shifted $\bm$-symbol of $\bla$.
Now since $e=\infty$, one can then show that 
the decomposition numbers $d_{\bmu,\bla}$ are non-zero only if $\bmu$ and $\bla$ are not comparable with respect to $\ll_\bm$,
which proves that there cannot be a basic set with respect to $\ll_\bm$.

Note also that explicit formulas are known for computing the elements of the canonical basis of the module $V(\br)$ in this case, see \cite{LeclercMiyachi2007},
which directly shows that all the elements appearing in the decomposition of $G_{\infty}(\bla,\br)$ have the same symbol up to a permutation of their elements.
\end{rem}

\medskip

\medskip

Putting together Propositions \ref{basicsetmregular}, \ref{basicsetmasymptotic} and \ref{basicsetmsingular}, we have proved:

\begin{thm}\label{theorem} Given a multicharge $\br\in\Z^l$ and a weight sequence $\bm\in\Q^l$, we have the following exhaustive classification:

\begin{itemize}
 \item If $\bm$ is regular, then $(\Hkn,\br)$ admits a canonical basic set with respect to $\ll_\bm$, namely the set of
$\si$-twisted Uglov $l$-partitions $\si(\Phi_{\bs^{\si^{-1}}}(n))$ where $\si$ is the $\bm$-adapted permutation and $\bs$ is the $\bm$-adapted multicharge
(cf. Proposition \ref{basicsetmregular}).
 \item If $\bm$ is asymptotic, then $(\Hkn, \br)$ admits a canonical basic set with respect to $\ll_\bm$, namely the set of
$\pi$-twisted Kleshchev $l$-partitions $\cK_{\cC_e}^\pi(n)$, where $\pi$ is the reordering permutation of the $\bm$-adapted multicharge 
(cf. Corollary \ref{corstabilisationklsi}).
 \item If $\bm$ is singular, then $(\Hkn,\br)$ does not admit any canonical basic set with respect to $\ll_\bm$.
\end{itemize}

\end{thm}

\begin{rem}
 Note that a weight sequence $\bm$ can be simultaneously regular and asymptotic. 
 In this case, one must have $\si(\Phi_{\bs^{\si^{-1}}}(n))=\cK_{\cC_e}^\pi(n)$, which is precisely what is stated in Remark \ref{remarkugsi=ug}.
\end{rem}

\medskip

Now that we have fully understood which values of $\bm$ yield a canonical basic set for $(\Hkn,\br)$ with respect to $\ll_\bm$,
we can state a similar result for the order induced by the $\ba$-function.
Indeed, by the compatibility property (\ref{compatibility}), if $\sB$ is the canonical basic set with respect to $\ll_\bm$, 
it is also the canonical basic set with respect to the $\ba$-function. 
Hence, the first two assertions in Theorem \ref{theorem} still hold for this order.
Further, on can prove using similar arguments that all of the results in the singular case also hold for this order. This leads to:

\begin{thm}\label{theorem0} $\ $ \vspace{0.0001cm}

\begin{itemize}
 \item If $\bm$ is regular, then $(\Hkn,\br)$ admits a canonical basic set with respect to the $\ba$-function, namely the set of
$\si$-twisted Uglov $l$-partitions $\si(\Phi_{\bs^{\si^{-1}}}(n))$ where $\si$ is the $\bm$-adapted permutation and $\bs$ is the $\bm$-adapted multicharge.
 \item If $\bm$ is asymptotic, then $(\Hkn, \br)$ admits a canonical basic set with respect to the $\ba$-function, namely the set of
$\pi$-twisted Kleshchev $l$-partitions  $\cK_{\cC_e}^\pi(n)$, where $\pi$ is the reordering permutation of the $\bm$-adapted multicharge.
 \item If $\bm$ is singular, then $(\Hkn,\br)$ does not admit any canonical basic set with respect to the $\ba$-function.
\end{itemize}

\end{thm}

\bigskip

\textbf{Acknowledgments:} I would like to thank C\'edric Lecouvey and Nicolas Jacon for their always appreciated comments and guidance, 
as well as J\'er\'emie Guilhot for some enlightening discussions.

\bigskip

\bibliographystyle{plain}
\bibliography{biblio}

\end{document}